\documentclass[a4paper,centertags,oneside,12pt]{amsart}
\usepackage{mathrsfs}
\usepackage{appendix}
\usepackage{amssymb}
\usepackage{fancyhdr}
\usepackage{charter}
\usepackage{typearea}
\usepackage{pdfsync}
\usepackage{color}
\usepackage[colorlinks,linkcolor=blue,anchorcolor=blue,citecolor=green]{hyperref}
\usepackage{mathrsfs}
\usepackage[a4paper,top=2.3cm,bottom=2.5cm,left=2.2cm,right=2.2cm]{geometry}

\usepackage{enumitem}
\usepackage{color}
\setlist[1]{itemsep=5pt}
\newcommand{\comment}[1]{}

\makeatletter
      \def\@setcopyright{}
      \def\serieslogo@{}
      \makeatother

\newcommand{\mbb}{\mathbb}

\newcommand{\ov}{\overline}

\newcommand{\La}{\Lambda}
\newcommand{\pa}{\partial}

\newcommand{\Om}{\Omega}
\newcommand{\al}{\alpha}
\newcommand{\be}{\beta}

\newcommand{\la}{\lambda}
\newcommand{\ti}{\tilde}

\newcommand{\norm}[1]{\left\Vert#1\right\Vert}
\newcommand{\abs}[1]{\left\vert#1\right\vert}

\newtheorem{theorem}{Theorem}[section]
\newtheorem{lemma}[theorem]{Lemma}
\newtheorem{corollary}[theorem]{Corollary}
\newtheorem{question}[theorem]{Question}
\newtheorem{proposition}[theorem]{Proposition}

\newtheorem{definition}[theorem]{Definition}
\newtheorem*{conjecture}{Conjecture}
\newtheorem{example}[theorem]{Example}
\newtheorem{remark}[theorem]{Remark}

\numberwithin{equation}{section}

\begin{document}
\title{The Invariant Szeg{\H{o}} metric on strongly pseudoconvex domains}
\keywords{Szeg\H o kernel, Szeg\H o metric, $L^2$-cohomology, K\"ahler-Einstein metric}
\subjclass{Primary: 32F45; Secondary: 32A25}
\author{Anjali Bhatnagar}
\address{School of Mathematics and Statistics, Wuhan University, Wuhan, Hubei 430072, China.}
\email{anj.bhatngr28@gmail.com}
\author{Jiliang Fan}
\address{School of Mathematics and Statistics, Wuhan University, Wuhan, Hubei 430072, China.}
\email{jiliangfan@whu.edu.cn}

\begin{abstract}
The Fefferman–Szeg\H{o} metric $g_{\operatorname{FS}}^\Om$ on a $C^\infty$-smooth bounded strongly pseudoconvex domain $\Omega\subset\mathbb{C}^n$ is an invariant metric defined via the Fefferman surface measure. 
For this metric, we first establish the vanishing of its $L^2$-Dolbeault cohomology on $\Omega$ unless $p+q=n$:
\[
\dim{H}_{2}^{p,q}(\Omega)=0~~\text{if }  p+q\neq n\quad\text{and}\quad 
\dim {H}_{2}^{p,q}(\Omega)=\infty\;\text{if}\;p+q=n.
\]
Furthermore, we prove that this metric has $C^\infty$-smooth bounded geometry. 
Using this analytic property, we obtain several rigidity results. 
Specifically, if the Fefferman–Szeg\H{o} metric is a gradient K\"ahler–Ricci soliton, then $\Omega$ is biholomorphic to the unit ball $\mathbb{B}^n$. 
Moreover, if the metric has constant scalar curvature, it must be Einstein, and again $\Omega$ is biholomorphic to $\mathbb B^n$. 
In addition, we give a Ramadanov-type criterion using the Fefferman-Szeg\H{o} invariant function. 
Finally, for $n=2$, under the existence of a K\"ahler immersion into a finite-dimensional ball that maps boundary to boundary transversally, we show that the logarithmic term of the Fefferman-Szeg\H{o} kernel vanishes to infinite order; consequently the boundary is locally spherical and, for simply connected case $\Om$ is biholomorphic to $\mathbb{B}^2$.
\end{abstract}
\maketitle

\section{Introduction}
The characterization of the unit ball $\mathbb{B}^n\subset \mathbb C^n$ among bounded domains is a central problem in complex geometry. A well‑known conjecture in this direction, due to Cheng \cite{c79}, states that if the Bergman metric $g_{\operatorname{B}}^\Om$ on a $C^\infty$-smooth bounded strongly pseudoconvex domain $\Omega\subset\mathbb C^n$ is K\"ahler–Einstein, then $\Omega$ is biholomorphic to $\mathbb B^n$. Cheng's conjecture was first confirmed in $\mathbb C^2$ by Fu–Wong \cite{fw97} and Nemirovski–Shafikov \cite{ns06}, and later established in all dimensions by Huang-Xiao \cite{hx21}. In \cite{hl23}, they generalized this result to Stein spaces with compact smooth strongly pseudoconvex boundary. Yau \cite{y82} posed a broader question: if the Bergman metric on a bounded pseudoconvex domain is K\"ahler–Einstein, must the domain be biholomorphic to a homogeneous domain? This question was recently answered by Savale-Xiao \cite{sx25} for finite type pseudoconvex domains in $\mathbb{C}^2$, and subsequently extended by Hsiao, Huang and Li \cite{hhl26} to real analytic pseudoconvex domains and $h$-extendible pseudoconvex domains in $\mathbb{C}^n$. Moreover, a very recent result in \cite{hjl25} shows that the Bergman metric of an unbounded pseudoconvex domain in $\mathbb{C}^n$ cannot be K\"ahler-Einstein if its boundary contains a non‑smooth strongly pseudoconvex polyhedral point. 

A crucial tool in the proof of Cheng's conjecture is Lu's uniformization theorem \cite{l66}, which states that if a bounded domain $\Omega$ admits a complete Bergman metric $g_{\operatorname{B}}^\Om$ of constant holomorphic sectional curvature, then $\Omega$ is biholomorphic to $\mathbb B^n$. This result was extended in \cite{dw25a} by relaxing the curvature hypothesis. For the incomplete Bergman metric, significant progress was made by Dong–Wong \cite{dw22, dw25b}, followed by works of Huang–Li \cite{hli} and Ebenfelt, Treuer, and Xiao \cite{etx25}.

Another direction concerns the existence of a Kähler immersion
\[
f:(\Omega,\lambda g_{\operatorname{B}}^\Omega)\longrightarrow (\mathbb{B}^N,g_{\operatorname{B}}^{\mathbb{B}^N}),
\qquad N\geq n.
\]
When \(N=n\), such a map is locally a biholomorphic isometry and hence preserves
holomorphic sectional curvature. For \(N>n\), it realizes the scaled Bergman metric as
the metric induced from a finite-dimensional complex hyperbolic ball, giving an extrinsic analogue of the constant holomorphic sectional curvature condition.  This viewpoint
was recently studied by Palmieri \cite{p25} for the Bergman metric.

 Further extensions of Cheng's conjecture were considered in \cite{s26b} in the context of constant scalar curvature (cscK). More precisely, if the Bergman metric $g_{\operatorname{B}}^\Om$ on a $C^2$-smooth bounded strongly pseudoconvex domain $\Omega$ has cscK, then it is K\"ahler–Einstein. Moreover, if $\pa\Om$ is $C^\infty$-smooth then $\Omega$ is biholomorphic to $\mathbb{B}^n$.

A natural extension of a K\"ahler–Einstein metric $g$ is the notion of a \emph{K\"ahler–Ricci soliton}, which satisfies 
\[
\operatorname{Ric}(g) + \mathcal{L}_X g = \lambda g
\]
for a real holomorphic vector field $X$ and a constant $\lambda \in \mathbb{R}$. If $X = \frac{1}{2}\nabla f$ for some $C^\infty$-smooth function $f$, then $g$ is called a \emph{gradient K\"ahler–Ricci soliton}. If $X$ is Killing, then $g$ is K\"ahler-Einstein. Sha \cite{s26a} proved that on bounded pseudoconvex domains, any K\"ahler–Ricci solitons with closed dual $1$-form are K\"ahler–Einstein provided the metric is complete, has $C^1$-bounded geometry and negative pinched Ricci curvature near the boundary. This provides a soliton analogue of Cheng's conjecture. 

Another classical approach to characterizing $\mathbb B^n$ is the Ramadanov conjecture \cite{R81}. It states that for a $C^\infty$-smooth bounded strongly pseudoconvex domain in $\mathbb{C}^n$, the boundary is locally spherical if the logarithmic term in Fefferman's asymptotic expansion of the Bergman kernel vanishes to infinite order at the boundary. 
This conjecture is known to imply Cheng's conjecture. The analogous conjecture for the Fefferman-Szeg\H o kernel first appeared in the manifold setting in \cite{lt04} and was later formulated for domains by Englis–Zhang \cite{ez10}:

\begin{conjecture}\label{conj:szego}
If the logarithmic term in the asymptotic expansion of the Fefferman-Szeg\H{o} kernel on a $C^\infty$-smooth bounded strongly pseudoconvex domain $\Omega \subset \mathbb{C}^n$ vanishes to infinite order at the boundary, then the boundary $\pa\Om$ is locally spherical.
\end{conjecture}
This conjecture is known to hold for $n=2$ by work of Boutet de Monvel, and was later improved by Ebenfelt \cite{E18} under the assumption of transverse symmetry. For $n\geq 3$, the conjecture remains open.

The Bergman and Szeg\H{o} kernels are two fundamental reproducing kernels in complex analysis. The Bergman metric induced from the Bergman kernel has been extensively studied, whereas the Szeg\H{o} metric defined similarly using the Szeg\H{o} kernel has received less attention until recent years. One reason for this is the lack of a nice transformation formula for the Euclidean surface measure $\sigma_{\operatorname{E}}$ under biholomorphisms (see \cite[Lemma 4.1]{wx22}), even for $C^\infty$-smooth bounded strongly pseudoconvex domains. To address this issue, Fefferman \cite[p. 259]{f79} introduced an invariant surface measure $\sigma_{\operatorname{F}}$ on the boundary $\partial\Omega$ of such domains, now known as the Fefferman surface measure. Barrett and Lee \cite{bl14} initiated a systematic study of the Fefferman--Szeg\H{o} metric on these domains -- defined via the Fefferman--Szeg\H{o} kernel, which uses $\sigma_{\operatorname{F}}$ as its boundary measure -- and investigated its relationship with the Bergman and Carath\'eodory metrics. Krantz studied the representative coordinates, analytic continuation, and completeness of the Fefferman–Szeg\H{o} metric in \cite{k19}, with further investigations in \cite{k21}. The boundary behaviour of the Fefferman-Szeg\H o metric and several associated invariants has been studied in \cite{m08,b26}. In one dimension, the Fefferman–Szeg\H{o} metric was studied in \cite{bb24} through its intrinsic properties, including geodesics (see \cite{b25} for higher dimension), curvature, and $L^2$-cohomology. Using harmonic analysis techniques from \cite{dlw25}, Yuan \cite{y25} extended the Fefferman surface area measure to $C^2$-smooth bounded domains in $\mathbb C^n$ provided the density function of $\sigma_{\operatorname{F}}$ (with respect to $\sigma_{\operatorname{E}}$) belongs to the Muckenhoupt class $A_2(\partial\Omega)$. We note that on a $C^2$-smooth bounded strongly pseudoconvex domain, the Fefferman surface area measure $\sigma_{\operatorname{F}}$ is well defined on the boundary $\partial\Omega$ and satisfies $\sigma_{\operatorname{F}}\approx \sigma_{\operatorname{E}}$ on $\partial\Omega$, so the techniques from \cite{y25} are not needed.

The following theorem extends the study of $L^2$-Dolbeault cohomology for the Fefferman–Szeg\H{o} metric to higher dimensions $n\geq 2$.

\begin{theorem}\label{l2-co}
Let $\Omega \subset \mathbb{C}^n$, $n\geq 2$, be a $C^\infty$-smooth bounded strongly pseudoconvex domain equipped with the Fefferman–Szeg\H{o} metric $g_{\operatorname{FS}}^\Omega$. Then, we have
\[
\dim H^{p,q}_{2}(\Omega)= \begin{cases}
0 & \text{if } p+q \neq n,\\
\infty & \text{if } p+q = n.
\end{cases}
\]
\end{theorem}
The $L^2$-cohomology for the Bergman metric on $C^\infty$-smooth bounded strongly pseudoconvex domains in $\mathbb{C}^n$ was studied by Donnelly–Fefferman \cite{df83} and Donnelly \cite{d94}, and in a more general setting by McNeal \cite{m02} and Ohsawa \cite{o89}, among others.

The next theorem collects several rigidity results for the Fefferman-Szeg\H{o} metric: the K\"ahler-Ricci soliton analogue of Cheng's conjecture, and the constant scalar curvature case.

\begin{theorem}\label{main:rigidity}
Let $\Omega \subset \mathbb{C}^n$ be a $C^\infty$-smooth bounded strongly pseudoconvex domain equipped with the Fefferman–Szeg\H{o} metric $g_{\operatorname{FS}}^\Omega$. Then the following hold:
\begin{enumerate}
\item[(i)] $g_{\operatorname{FS}}^\Omega$ has $C^\infty$-smooth bounded geometry.
\item[(ii)] $g_{\operatorname{FS}}^\Omega$ is a gradient K\"ahler--Ricci soliton if and only if $\Omega$ is biholomorphic to $\mathbb B^n$.
\item[(iii)] If $g_{\operatorname{FS}}^\Omega$ has constant scalar curvature, then it is Einstein. Moreover, $\Omega$ is biholomorphic to $\mathbb B^n$.
\end{enumerate}
\end{theorem}
Part (i) provides the key analytic property of the Fefferman–Szeg\H{o} metric. Part (ii) relies on a general result of Sha \cite[Theorem 1.1]{s26a}, which requires the metric to have $C^1$-bounded geometry, a condition guaranteed by part (i). Part (iii) is motivated by \cite[Theorem 1.4 (2)]{s26b} on the Bergman metric. Yuan \cite{y25} first established the K\"ahler--Einstein rigidity for the Fefferman--Szeg\H{o} metric, i.e., if $g_{\operatorname{FS}}^\Omega$ is K\"ahler--Einstein then $\Omega$ is biholomorphic to $\mathbb B^n$. The above results extend Yuan's theorem in two directions: to K\"ahler--Ricci solitons and to constant scalar curvature.
\begin{question}
    It is natural to ask whether the $C^\infty$-smoothness assumption on the boundary can be relaxed in Theorem \ref{main:rigidity}.
\end{question}
To prove Cheng's conjecture for the Fefferman–Szeg\H o metric, Yuan \cite{y25} used a result of Fu–Wong \cite{fw97} to establish Conjecture \ref{conj:szego}. Inspired by \cite[Theorem 5.13]{m21}, we weaken the assumption the K\"ahler--Einstein assumption to an asymptotic K\"ahler--Einstein condition near the boundary and prove Conjecture \ref{conj:szego} without relying on the Fu–Wong's result. We hope that our result will shed further light on Conjecture~\ref{conj:szego}.

\begin{theorem}\label{gen_rama_conj}
Let $\Omega=\{z\in\mathbb C^n:\rho(z)<0\}$, $n\geq 4$, be a $C^\infty$-smooth bounded strongly pseudoconvex domain with a $C^\infty$-smooth strongly plurisubharmonic defining function $\rho$. Let $\mathfrak{R}_\Omega(z)$ be the Fefferman–Szeg\H{o} invariant function. Then the boundary $\partial\Omega$ is locally spherical if and only if for every $p\in\partial\Omega$,
\[
\lim_{z\to p}\rho(z)^{-2}\left(\mathfrak{R}_\Omega(z)-C_n\right)=0,
\] where $C_n=((n-1)!)^{-\frac{n+1}{n}}n^n\pi^{n+1}$.
\end{theorem}
The proof of Theorem \ref{gen_rama_conj} relies on the asymptotic expansion of the Fefferman–Szeg\H o kernel, which creates an obstruction in dimensions $2$ and $3$. Therefore, the following question remains open:

\begin{question}
For $n = 2$ or $n = 3$, does the Fefferman–Szeg\H o invariant characterize local sphericity of the boundary of a $C^\infty$-smooth bounded strongly pseudoconvex domain in $\mathbb{C}^n$?
\end{question}


Finally, inspired by \cite[Theorem 1]{p25} for the Bergman metric; we obtain a result connecting the Fefferman–Szeg\H{o} metric with the local sphericity problem for $C^2$-smooth bounded strongly pseudoconvex domains in dimension $n=2$.

\begin{theorem}\label{kah_imm_ball}
Let $\Omega \subset \mathbb{C}^n$ be a $C^2$-smooth bounded strongly pseudoconvex domain. 
Assume that there exists $\lambda>0$ such that $\frac{N}{\lambda}-n\in\mathbb{N}$ and a K\"ahler immersion
\[
f \colon (\Omega,\lambda\,g_{\operatorname{FS}}^\Omega)\longrightarrow (\mathbb{B}^N, g_{\operatorname{FS}}^{\mathbb B^N}), 
\qquad N<\infty,
\]
which extends smoothly to a map $F\colon U\to V$, where $U$ and $V$ are neighbourhoods of 
$\overline{\Omega}$ and $\overline{\mathbb{B}^N}$, respectively. 
Suppose further that
\[
F(\partial\Omega)\subset \partial\mathbb{B}^N
\quad\text{transversally, and}\quad
F\bigl(U\setminus\overline{\Omega}\bigr)\subset V\setminus\overline{\mathbb{B}^N}.
\]
Then the logarithmic coefficient in the asymptotic expansion of the Fefferman-Szeg\H{o} kernel vanishes to infinite order on $\partial\Omega$. 
For $n=2$, this implies that $\partial\Omega$ is locally spherical. 
In addition, if $\Omega\subset\mathbb{C}^2$ is simply connected, then $\Omega$ is biholomorphic to $\mathbb{B}^2$.
\end{theorem}
\begin{remark}
   Since $\Omega$ is bounded, one can always construct a holomorphic immersion of $(\Omega,g_{\operatorname{FS}}^\Omega)$ into the complex hyperbolic space $F(N,b)$ (see Section \ref{kahler immersion} for definition) for $N\ge n$ and $b<0$. 
Hence the hypothesis of Theorem~\ref{kah_imm_ball} is natural.
\end{remark}
Note that under the assumptions of Theorem \ref{kah_imm_ball}, the hypothesis of Conjecture \ref{conj:szego} is satisfied. Hence, if Conjecture \ref{conj:szego} holds, then $\Omega$ is biholomorphic to $\mathbb{B}^n$ in any dimension.

\subsection*{Organization of the paper}
In Section~2, we recall the necessary background for the paper. In Section~3, we study the $L^2$-Dolbeault cohomology of the Fefferman--Szeg\H o metric and prove Theorem~\ref{l2-co}. Section~4 is devoted to proving the rigidity results in Theorem~\ref{main:rigidity} for the Fefferman--Szeg\H o metric. In Section~5, we establish a Ramadanov-type criterion (Theorem~\ref{gen_rama_conj}) in terms of the Fefferman--Szeg\H{o} invariant function. Finally, in Section~6, we study K"ahler immersions into finite-dimensional balls and prove Theorem~\ref{kah_imm_ball}.

\subsection*{Notations}
\begin{itemize}
    \item For two nonnegative functions \(f\) and \(g\), we write \(f\approx g\) if there exists a constant \(C>1\) such that
    \[
    C^{-1}g\leq f\leq Cg.
    \]

    \item We use \(O(\cdot)\) and \(o(\cdot)\) in the usual asymptotic sense as \(z\to\partial\Omega\), where the region in which \(z\) varies will be clear from the context. Thus \(f=O(g)\) means that there exists a constant \(C>0\) such that
    \[
    |f(z)|\leq C|g(z)|
    \]
    near \(\partial\Omega\), while \(f=o(g)\) means that
    \[
    \frac{f(z)}{g(z)}\to 0
    \quad \text{as } z\to\partial\Omega.
    \]
\end{itemize}
\section{Preliminaries}
Let $\Omega = \{z\in\mathbb{C}^n : \rho(z) < 0\} \subset \mathbb{C}^n$ be a $C^2$-smooth bounded strongly pseudoconvex domain, and let $\rho$ be a $C^2$-smooth strongly plurisubharmonic defining function.

The \textit{Fefferman surface area measure} $\sigma_{\operatorname{F}}$ on the boundary $\partial\Omega$ is defined by  
\[
d\sigma_{\operatorname{F}}^{\partial\Omega} \wedge d\rho = J[\rho]^{\frac{1}{n+1}}\, dV,
\qquad\text{or equivalently}\qquad
d\sigma_{\operatorname{F}}^{\partial\Omega} =  \frac{J[\rho]^{\frac{1}{n+1}}}{\|d\rho\|}\, d\sigma_{\operatorname{E}},
\]
where $\sigma_{\operatorname{E}}$ is the usual Euclidean surface area measure, $\|d\rho\|$ is the Euclidean norm of the gradient $\nabla\rho$ of $\rho$, and  
\[
J[\rho] = (-1)^n \det\begin{bmatrix}
\rho & \partial_{\bar j}\rho\\[2pt]
\partial_i\rho & \partial_{i\bar j}\rho
\end{bmatrix}_{ i,j=1
}^n
\]
where $\partial_i=\partial/\partial z_i,\; \partial_{\bar j}=\partial/\partial\bar z_j,\;
\partial_{i\bar j}=\partial_i\partial_{\bar j}$. 

The measure $\sigma_{\operatorname{F}}$ is independent of the choice of $\rho$ and satisfies $d\sigma_{\operatorname{F}} \approx d\sigma_{\operatorname{E}}$. 

\subsection*{(T)-biholomorphisms}
A biholomorphism $F:\Omega_1\to\Omega_2$ between two such domains is called an \textit{(T)-biholomorphism} if  
\begin{itemize}
    \item[(i)] $F$ extends to a $C^2$-diffeomorphism $\overline{\Omega}_1\to\overline{\Omega}_2$, and  
    \item[(ii)] $(\det J_{\mathbb C}F)^{\frac{n}{n+1}}$ admits a globally defined holomorphic branch on $\Omega_1$.
\end{itemize}
Under such a map, the Fefferman surface area measure transforms as  
\begin{equation}\label{pullback:feffer_surf_meas}
   F^*\big(d\sigma_{\operatorname{F}}^{\partial\Omega_2}\big) = |\det J_{\mathbb{C}}F(z)|^{\frac{2n}{n+1}}\, d\sigma_{\operatorname{F}}^{\partial\Omega_1}
\end{equation}

(see \cite[Proposition~1]{bl14} for further details).
\subsection*{Fefferman-Szeg\H o kernel}
Let $L^2_{\operatorname{F}}(\partial\Omega)$ be the space of square‑integrable functions with respect to $d\sigma_{\operatorname{F}}$, and set $A(\Omega)=\mathcal{O}(\Omega)\cap C(\overline{\Omega})$. The \textit{Fefferman–Hardy space} $\mathrm{H}^2(\partial\Omega)$ is the closure of $A(\Omega)$ in $L^2_{\operatorname{F}}(\partial\Omega)$. Since $d\sigma_{\operatorname{F}}\approx d\sigma_{\operatorname{E}}$, $\mathrm{H}^2(\partial\Omega)$ coincides with the classical Hardy space as a set, but its Hilbert space structure differs. The space $\mathrm H^2(\partial\Omega)$ is a reproducing kernel Hilbert space. 
Its reproducing kernel $S_\Omega(z,w)$ is called the Fefferman-Szeg\H{o} kernel and is characterized as follows: for each $z\in\Omega$, $S_\Omega(\cdot,z)\in \mathrm{H}^2(\partial\Omega)$; for all $z,w\in\Om$, $S_\Omega(z,w)=\overline{S_\Omega(w,z)}$; for every $f\in \mathrm{H}^2(\partial\Omega)$ and $z\in\Om$, \[\displaystyle f(z)=\int_{\partial\Omega} f(w)S_\Omega(z,w)\,d\sigma_{\operatorname{F}}(w).\]

If $\{\phi_j\}$ is any complete orthonormal basis of $\mathrm{H}^2(\partial\Omega)$, then  
\begin{equation}\label{ser-exp-sgoker}
    S_\Omega(z,w)=\sum_{j=1}^\infty \phi_j(z)\overline{\phi_j(w)},
\end{equation}
where the series converges uniformly on compact subsets of $\Omega\times\Omega$.

\begin{proposition}\label{trans:feff_sgo_kers}
    Let $F:\Omega_1\to\Omega_2$ be an (T)-biholomorphism. Then for all $z,w\in\Om_1$, 
\[
S_{\Omega_1}(z,w)
=S_{\Omega_2}(F(z),F(w))
\det J_{\mathbb{C}}F(z)^{\frac{n}{n+1}}
\overline{\det J_{\mathbb{C}}F(w)}^{\frac{n}{n+1}}.
\]
\end{proposition}
\begin{proof}
    For $i=1,2$, let $\mathcal{D}_i=A(\Omega_i)$. Define $U:\mathcal{D}_2\to\mathcal{D}_1$ by
\[
(Uf)(z)=f(F(z))\;\det J_{\mbb C}F(z)^{\frac{n}{n+1}},\qquad z\in\Omega_1.
\]
For $f\in\mathcal{D}_2$, using (\ref{pullback:feffer_surf_meas}) and the change of variables $\eta=F(\zeta)$,
\[
\|Uf\|_{L_{\operatorname{F}}^2(\partial\Omega_1)}^2
= \int_{\partial\Omega_1}|f(F(\zeta))|^2\,|\det J_{\mbb C}F(\zeta)|^{\frac{2n}{n+1}}\,d\sigma_{\operatorname{F}}(\zeta)
= \int_{\partial\Omega_2}|f(\eta)|^2\,d\sigma_{\operatorname{F}}(\eta)
= \|f\|_{L_{\operatorname{F}}^2(\partial\Omega_2)}^2.
\]
Hence $U$ is an isometry. Because $F^{-1}$ also satisfies Condition~(T), $U$ maps $\mathcal{D}_2$ onto $\mathcal{D}_1$. Consequently, $U$ extends uniquely to a unitary operator
\[
U: \mathrm{H}^2(\partial\Omega_2,d\sigma_{\operatorname{F}}^{\pa\Om_2})\longrightarrow \mathrm{H}^2(\partial\Omega_1,d\sigma_{\operatorname{F}}^{\pa\Om_1}).
\]
Let $\{\psi_j\}_{j\ge1}$ be a complete orthonormal basis of $\mathrm{H}^2(\partial\Omega_2,d\sigma_{\operatorname{F}})$. Then $\{U\psi_j\}_{j\ge1}$ is a complete orthonormal basis of $\mathrm{H}^2(\partial\Omega_1,d\sigma_{\operatorname{F}})$. Using (\ref{ser-exp-sgoker}), we have
\[
S_{\Omega_1}(z,w)=\sum_{j=1}^\infty (U\psi_j)(z)\;\overline{(U\psi_j)(w)},\qquad z,w\in\Omega_1.
\]
Substituting $(U\psi_j)(z)=\psi_j(F(z))\,(\det J_{\mbb C}F(z))^{n/(n+1)}$ gives
\begin{multline*}
    S_{\Omega_1}(z,w)=\det J_{\mbb C}F(z)^{\frac{n}{n+1}}\;
\overline{\det J_{\mbb C}F(w)}^{\frac{n}{n+1}}\;
\sum_{j=1}^\infty \psi_j(F(z))\overline{\psi_j(F(w))}\\
=\det J_{\mbb C}F(z)^{\frac{n}{n+1}}\;
\overline{\det J_{\mbb C}F(w)}^{\frac{n}{n+1}}S_{\Omega_2}(F(z),F(w)),
\end{multline*}
which completes the proof.
\end{proof}
\subsection*{Fefferman–Szeg\H{o} metric}
Let $S_\Om(z):=S_\Om(z,z)$ denote the restriction of the Fefferman-Szeg\H{o} kernel to the diagonal. It can be seen that $\log S_\Om(z)$ is a $C^\infty$-smooth strongly plurisubharmonic function on $\Om$, which induces the K\"ahler metric defined as
\[
g_{\operatorname{FS}}^\Omega = \sum_{i,j=1}^n g_{i\bar j}^\Omega(z)\, dz^i d\bar z^j,
\qquad 
g_{i\bar j}^\Omega(z)=\frac{\partial^2}{\partial z_i\partial\bar z_j}\log S_\Omega(z),
\]
called the \textit{Fefferman–Szeg\H o metric}. By Proposition \ref{trans:feff_sgo_kers}, this metric is invariant under (T)-biholomorphisms, i.e., for all $z\in\Om,~~X\in\mbb C^n$,  \[g^{\Om_1}_{\operatorname{FS},z}\big(X,X\big)=g^{\Om_2}_{\operatorname{FS},F(z)}\big(J_{\mbb C}F(z)X,J_{\mbb C}F(z)X\big).\] By \cite{bl14,y25}, the Fefferman--Szeg\H{o} metric dominates the Carath\'eodory metric. Consequently, $(\Om,g_{\operatorname{FS}}^\Om)$ is a complete K\"ahler manifold. 

Let $G_\Omega(z)=\big[g_{i\bar j}^\Omega(z)\big]_{ i,j=1
}^n$ be the matrix induced from $g^\Om_{FS}$ and $G_\Omega^{-1}=\big[g_\Omega^{i\bar j}(z)\big]_{ i,j=1
}^n$ its inverse. The \textit{Fefferman–Szeg\H o invariant} is defined by
\[
\mathfrak{R}_\Omega(z)=\frac{\det G_\Omega(z)}{S_\Omega(z)^{\frac{n+1}{n}}}.
\]
By Proposition~\ref{trans:feff_sgo_kers}, $\mathfrak R_\Omega$ is invariant under $(T)$-biholomorphisms, i.e., $\mathfrak{R}_{\Om_2}\big(\Phi(z)\big)=\mathfrak{R}_{\Om_1}(z).$

The Ricci curvature of $g^\Om_{FS}$ is given by
\[
\operatorname{Ric}(g_{\operatorname{FS}}^\Omega)=\operatorname{Ric}_{i\bar j}^\Omega\, dz^i d\bar z^j,\qquad
\operatorname{Ric}_{i\bar j}^\Omega=-\partial_{i\bar j}\log\det G_\Omega(z).
\]
Here and henceforth, we use the Einstein summation convention. 

The metric $g_{\operatorname{FS}}^\Omega$ is said to be \textit{K\"ahler–Einstein} if there exists $\la\in\mathbb R$ such that \[\operatorname{Ric}(g_{\operatorname{FS}}^\Omega)=\lambda g_{\operatorname{FS}}^\Omega.\] When $\pa\Om$ is $C^\infty$-smooth, one has $\lambda=-(n+1)/n$ (see \cite{b26}). Moreover, if $\mathfrak{R}_\Omega\equiv C_n$ then $g_{\operatorname{FS}}^\Omega$ is K\"ahler–Einstein (see \cite{y25}) and $C_n=\left((n-1)!\right)^{-\frac{n+1}{n}} n^n \pi^{n+1}$ (see \cite{b26,y25}).

The scalar curvature of $g^\Om_{FS}$ is defined by
\[\mathrm{Sca}(g_{\operatorname{FS}}^\Omega)=\operatorname{tr}_{g_{\operatorname{FS}}^\Om} \left(\operatorname{Ric}(g_{\operatorname{FS}}^\Om)\right)=g_\Omega^{i\bar j}\operatorname{Ric}_{i\bar j}^\Omega.\]
For a $C^\infty$ function $f:\Omega\to\mathbb{R}$, the gradient and Laplacian with respect to $g^\Om_{FS}$ are  given by 
\[
\nabla f = g_\Omega^{i\bar j}\big(\partial_i f\,\partial_{\bar j}+\partial_{\bar j}f\,\partial_i\big),\qquad
\Delta_{g_{\operatorname{FS}}^\Omega}f = g_\Omega^{i\bar j}\,\partial_{i\bar j}f.
\]
Let $\nabla$ denote the Levi-Civita connection of $g_{\operatorname{FS}}^\Om$ and the covariant derivatives are wriiten as $\nabla_i = \nabla_{\partial_i}\text{ and } \nabla_{\bar{j}}= \nabla_{\partial_{\bar{j}}}.$ Since the mixed Christoffel symbols vanish for K\"ahler metrics, we have $\nabla_i\nabla_{\bar j}f=\partial_{i\bar j}f$.

The metric $g_{\operatorname{FS}}^\Omega$ is said to be a K\"ahler-Ricci soliton if there exists a real holomorphic vector field $X$ and $\la\in\mathbb R$ such that 
\[\operatorname{Ric}(g_{\operatorname{FS}}^\Omega)+\mathcal{L}_Xg^\Om_{\operatorname{FS}}=\la g^\Om_{\operatorname{FS}}.\] A vector field $X=X^{1,0}+X^{0,1}$ is called real if $X^{0,1}=\overline{X^{1,0}}$ and it is holomorphic if its $(1,0)$-part $X^{1,0}:=(X-\sqrt{-1}JX)/2$ is holomorphic with respect to the standard complex structure $J$ of $\mathbb C^n$. If there exists a $C^\infty$ real-valued function $f$ such that $X=\frac{1}{2}\nabla f$, then $g^\Om_{\operatorname{FS}}$ is called a \textit{gradient K\"ahler-Ricci soliton} with potential $f$. 
 Moreover, the condition that $\nabla f$ be a real holomorphic vector field is equivalent to $\nabla_i \nabla_j f=0$ for all $1\leq i,j \leq n$. 
 
 For each $z\in\Omega$, the pointwise norm of the $1$-form $\eta$ with respect to
$g_{\operatorname{FS}}^\Omega$ is defined by
\[
|\eta|_{g_{\operatorname{FS},z}^\Omega}
:=
\sup\left\{
|\eta_z(X)|:\, X\in\mathbb C^n,\ 
g_{\operatorname{FS},z}^\Omega(X,X)\leq 1
\right\}.
\]
We say that $\eta$ is bounded in the supremum norm if
\[
\|\eta\|_{g_{\operatorname{FS}}^\Omega}
:=
\sup_{z\in\Omega}|\eta|_{g_{\operatorname{FS},z}^\Omega}<\infty.
\]


\section{\texorpdfstring{$L^2$-cohomology of the Fefferman-Szeg\H o metric}{L2-cohomology of the Fefferman-Szego metric}}
Throughout this section, let $\Om\subset\mathbb C^n$
be a bounded strongly pseudoconvex domain with
 $C^\infty$-smooth boundary, where $n\geq 2$, and let $\rho$
be a $C^\infty$-smooth strongly plurisubharmonic defining function for $\Om$. Let $\Lambda_2^k(\Omega)$ denote the space of $k$-forms on $\Omega$ that are square-integrable with respect to the Fefferman-Szeg\H o metric $g_{\operatorname{FS}}$. 
 The $L^2$-cohomology groups of the complex
\[
\La_2^0(\Om) \xrightarrow{d_0} \La_2^1(\Om) \xrightarrow{d_1} \La_2^2(\Om) \xrightarrow{d_2}
\cdots \xrightarrow{d_{2n-1}}\La_{2}^{2n}(\Omega)\xrightarrow{d_{2n}}0\]
are defined by
\[
H_{2}^k(\Omega) := \ker d_k \, / \, \overline{\mathrm{Im}(d_{k-1})},
\]
where the closure is taken with respect to the $L^2$-norm induced by $g_{\operatorname{FS}}$. Since $g_{\operatorname{FS}}$ is complete, every $L^2$-cohomology class has a unique harmonic representative. Hence
\[
H_2^k(\Omega) \cong \mathcal{H}_2^k(\Omega),
\]
where $\mathcal{H}_2^k(\Omega)$ denotes the space of square‑integrable harmonic $k$-forms. Consequently, for each bidegree $(p,q)$ we have
\[
H_2^{p,q}(\Omega) \cong \mathcal{H}_2^{p,q}(\Omega),
\]
and the space of harmonic forms decomposes as
\[
\mathcal{H}_2^k(\Omega) = \bigoplus_{p+q=k} \mathcal{H}_2^{p,q}(\Omega).
\]
The goal of this section is to prove Theorem~\ref{l2-co}, i.e.,
\[
\dim H_2^{p,q}(\Omega)=0 \;\text{ if }\; p+q\neq n,\qquad
\dim H_2^{p,q}(\Omega)=\infty \;\text{ if }\; p+q=n.
\] 
\subsection*{Vanishing outside the middle dimension}
We begin with the following vanishing theorem of Donnelly \cite{d94} concerning the $L^2$-cohomology outside the middle dimension.

\begin{theorem}\label{donelly}
Let $(M,ds^2)$ be a complete K\"ahler manifold of complex dimension $n$ with K\"ahler form $\omega$. If $\omega = d\eta$ with $\eta$ bounded in supremum norm, then
\[
\mathcal{H}_2^k(M) = 0 \quad \text{for all } k \neq n.
\]
\end{theorem}
To apply this to $g_{\operatorname{FS}}$, we set $\eta = \Theta$ where
\[
\Theta := -i \pa_i\log S(z)\,dz^i\qquad\text{so that}\qquad \omega = d\Theta,
\]
where we used
\(
\bar\partial\partial=-\partial\bar\partial.
\)
To apply Theorem~\ref{donelly}, we must verify that $\Theta$ is bounded with respect to $g_{\operatorname{FS}}$.

Let $\varepsilon(z) = -\rho(z)$ and define
\[
h(z) = S(z)\varepsilon(z)^n,\qquad \mathfrak h(z) = \log h(z),
\]
so that \(\log S(z) = -n\log\varepsilon(z) + \mathfrak h(z)\). A straightforward computation gives
\begin{align}
\Theta_z(X) &= -i\Bigl(\langle X,\bar\partial\mathfrak h(z)\rangle + n\,\frac{\langle X,\bar\partial\rho(z)\rangle}{\varepsilon(z)}\Bigr), \label{theta_expr}\\
g_{\operatorname{FS},z}(X,X) &= n\,\frac{|\langle X,\bar\partial\rho(z)\rangle|^2}{\varepsilon(z)^2}
+ \frac{n}{\varepsilon(z)}\,L_{\rho(z)}(X,X) + L_{\mathfrak h(z)}(X,X). \label{metric_expr}
\end{align}

Here, $\ov\pa f=(\pa_{\ov 1}f,\cdots, \pa_{\ov n}f)$ and $L_f$ denotes the Levi form of $f$. The following asymptotic estimates for $\mathfrak h$ follow from the asymptotic expansion of $S(z)$ (see \cite{b25}).

\begin{lemma}\label{asymp}
For \(n\ge2\), we have
\[
\partial_i\mathfrak h = O(1),\qquad
\partial_{i\bar j}\mathfrak h = 
\begin{cases}
O(\log\varepsilon^{-1}) & n=2,\\
O(1) & n\ge 3.
\end{cases}
\]
\end{lemma}
Using the above estimates, we prove the following lemma.
\begin{lemma}\label{theta_bounded}
The $1$-form $\Theta$ is bounded in the supremum norm, i.e.,
\[
\|\Theta\|_{g_{\operatorname{FS}}}=\sup_{z\in\Omega}\;\sup_{X\in\mathbb{C}^n\setminus\{0\}}\frac{|\Theta_z(X)|^2}{g_{\operatorname{FS},z}(X,X)} < \infty.
\]  
\end{lemma}

\begin{proof}
By the homogeneity of $\Theta_z(X)$ and $g_{\operatorname{FS},z}(X,X)$, we can fix $|X|=1$ throughout the proof. It suffices to consider points $z$ near $\partial\Omega$, because on compact subsets the ratio is clearly bounded. Now let $X \in \mathbb{C}^n$ be any. For each $p \in \partial\Omega$, the tangent space admits a canonical splitting $\mathbb{C}^n = H_p(\partial\Omega) \oplus N_p(\partial\Omega)$ into complex tangential and normal directions at $p$; therefore $X$ can be uniquely written as $X = X_T(p) + X_N(p)$ with $X_T(p) \in H_p(\partial\Omega)$ and $X_N(p) \in N_p(\partial\Omega)$. By the smoothness of $\partial\Omega$, for every point $z$ sufficiently close to $\pa\Om$ there exists a unique nearest boundary point $\pi(z) \in \partial\Omega$ satisfying $|z - \pi(z)| = \operatorname{dist}(z, \partial\Omega)$. Consequently, we  decompose $X$ relative to $\pi(z)$: $X = X_T(\pi(z)) + X_N(\pi(z))$. For brevity, we set $X_T(z) := X_T(\pi(z))$ and $X_N(z) := X_N(\pi(z))$, and refer to them as the tangential and normal components of $X$ at $z$, respectively.

If $\langle X,\bar\partial\rho(p)\rangle \neq 0$, then using Lemma \ref{asymp} in \eqref{theta_expr} and \eqref{metric_expr}, we have $|\Theta_z(X)|^2/g_{\operatorname{FS},z}(X,X)$ converges to $n$ as $z\to p$. 

If $\langle X,\bar\partial\rho(p)\rangle = 0$, then both $\varepsilon(z)^2|\Theta_z(X)|^2$ and $\varepsilon(z)^2g_{\operatorname{FS},z}(X,X)$ tend to zero as $z\to p$; therefore the elementary limit argument used for non-tangential directions fails. Thus, we examine the ratio for $z$ sufficiently close to $\partial\Omega$.

Using Lemma~\ref{asymp} in \eqref{theta_expr} and \eqref{metric_expr} we obtain
\begin{equation}\label{approx:norm_X_theta}
    |\Theta_z(X)|^2 = \frac{n^2\,|\langle X,\overline\partial\rho(z)\rangle|^2}{\varepsilon(z)^2} + O\!\left(\frac{1}{\varepsilon(z)}\right),
\end{equation}
\[
g_{\operatorname{FS},z}(X,X) = \frac{n\,|\langle X,\overline\partial\rho(z)\rangle|^2}{\varepsilon(z)^2} + \frac{n}{\varepsilon(z)}\,L_{\rho(z)}(X,X) + O\!\bigl(\log\varepsilon^{-1}\bigr).
\]
The two singular terms blow up at different rates: $\varepsilon^{-2}$ (normal direction) and $\varepsilon^{-1}$ (tangential direction). The logarithmic term is of lower order and does not affect the leading blow‑up. To compare the relative strengths of the two singularities, we consider the ratio
\[
\frac{ \frac{n|\langle X,\overline\partial\rho(z)\rangle|^2}{\varepsilon(z)^2} }{ \frac{n}{\varepsilon(z)}L_{\rho(z)}(X,X) } = \frac{|\langle X,\overline\partial\rho(z)\rangle|^2}{\varepsilon(z)\,L_{\rho(z)}(X,X)}\approx\frac{|\langle X,\overline\partial\rho(z)\rangle|^2}{\varepsilon(z)}. 
\]
The quantity $|\langle X,\overline\partial\rho(z)\rangle|^2/\varepsilon(z)$ determines which singularity dominates.
 Accordingly, we set
\[
\nu(z) = \frac{|\langle X,\overline\partial\rho(z)\rangle|}{\varepsilon(z)},\qquad 
\tau(z) = \frac{1}{\varepsilon(z)}.
\]
 Choose a constant $C$ large enough, for instance
\[
C > 2M=2\sup_{z\in\overline\Omega}\;\sup_{|X|=1} L_{\rho(z)}(X,X),
\]
which will later allow uniform estimates. We now split the analysis into two disjoint regimes sufficiently near $\pa\Om$ as follows:

\begin{itemize}
\item \textbf{Case N:} $\nu(z)^2 \ge C\tau(z)$---the normal singularity dominates.
\item \textbf{Case T:} $\nu(z)^2 < C\tau(z)$---the tangential singularity dominates.
  
\end{itemize}

\medskip

\emph{Case N}: $\nu(z)^2 \ge C\tau(z)$. We will first establish the lower bound of $g_{\operatorname{FS}}$ near $\pa\Om$. Using \eqref{metric_expr}, we write $g_{\operatorname{FS}}$ in terms of $\tau$ and $\nu$ as
\[g_{\operatorname{FS},z}(X,X)=n\nu(z)^2+n\tau(z)L_{\rho(z)}(X,X)+L_{\mathfrak{h}(z)}(X,X).\]
Using $|L_\rho(X,X)|\le M$, we obtain
\begin{equation*}
    \left| \tau(z)L_{\rho(z)}(X,X) \right|\le \frac{M}{C}\nu(z)^2.
\end{equation*}
This implies 
\[g_{\operatorname{FS},z}(X,X)\geq n\left(1-\frac{M}{C}\right)\nu(z)^2+L_{\mathfrak{h}(z)}(X,X).\]

Using Lemma \ref{asymp} and $\nu(z)^2 \ge C\tau(z)$, we get
\begin{equation}\label{N:Leviform_h}
    |L_{\mathfrak h(z)}(X,X)| =O(|\log\varepsilon(z)|)=o\big(\nu(z)^2\big).
\end{equation}

Thus (in this regime), we have
\begin{equation}\label{N:lowerbdd:feffsgo_metric}
    g_{\operatorname{FS},z}(X,X) \ge \frac{n}{2}\Bigl(1-\frac{M}{C}\Bigr)\nu(z)^2
\end{equation}
for $z$ sufficiently close to $\pa\Om$. 

From (\ref{approx:norm_X_theta}), the numerator satisfies
\[
|\Theta_z(X)|^2 = n^2\nu(z)^2 + O(\nu(z)).
\]
Therefore, the ratio $|\Theta_z(X)|^2/g_{\operatorname{FS},z}(X,X)$ is uniformly bounded from above in this regime.

\medskip

\emph{Case T:} $\nu(z)^2 < C\tau(z)$.
Then $|\langle X,\bar\partial\rho(z)\rangle| \le \sqrt{C\varepsilon(z)}$, so the normal component \(|X_N(z)|\) is of order $\sqrt{\varepsilon(z)}$, i.e., $|X_N(z)|=O(\sqrt{\varepsilon(z)})$. By $|X|=1$, the tangential component $|X_T(z)|$ is bounded below by $1/2$ for small $\varepsilon(z)$, and thus the Levi form $L_\rho(X,X)$ has a uniform positive lower bound. From Lemma \ref{asymp}, we have 
\begin{equation}\label{T:Leviform:h}
    \varepsilon(z)^2 L_{\mathfrak h(z)} = o(\varepsilon(z)).
\end{equation} Combining the above estimates yields
\begin{equation}\label{T:lowerbdd:feffsgo_metric}
    \varepsilon(z)^2 g_{\operatorname{FS},z}(X,X) \ge \frac{n c_0}{32}\,\varepsilon(z)
\end{equation}
for some $c_0>0$. On the other hand $\varepsilon(z)^2|\Theta_z(X)|^2 \le C_2\varepsilon(z)$. Thus the ratio is again uniformly bounded from above in this regime.

Both cases give a uniform upper bound for the ratio near the boundary. Hence $\Theta$ is bounded in the supremum norm.  
\end{proof}

Since $g_{\operatorname{FS}}$ is complete and $\omega = d\Theta$ with $\Theta$ bounded, Theorem \ref{donelly} immediately yields
\[
H_2^k(\Omega) = 0 \quad \text{for all } k \neq n.
\]

\subsection*{Infinite dimensionality in the middle dimension}
To establish the infinite dimensionality in the middle degree, we use the following theorem of Ohsawa \cite{o89}.

\begin{theorem}[Ohsawa]\label{ohsawa}
Let \(D\subset\mathbb{C}^n\) be a domain with a Hermitian metric \(ds^2\). Assume there exists a non‑degenerate regular boundary point \(z_0\in\partial D\) and, in a neighbourhood \(U\) of \(z_0\), a defining function \(\phi\) (so that \(D\cap U = \{\phi<0\}\)) and a Hermitian metric \(ds_U^2\) on \(U\) such that
\[
C^{-1}ds^2 < (-\phi)^{-a}ds_U^2 + (-\phi)^{-b}\partial\phi\,\overline{\partial\phi} < C\,ds^2
\]
on \(U\cap D\) for some constants \(C>0\) and \(1\le a\le b<a+3\). Then for every pair \((p,q)\) with \(p+q=n\) we have \(\dim H_2^{p,q}(D)=\infty\).
\end{theorem}

To apply this to \(\Omega\), we establish the following comparison.

\begin{lemma}\label{comparison}
Let $\rho$ be a $C^\infty$-smooth defining function of $\Omega$. Then
\[
g_{\operatorname{FS}} \;\approx\; (-\rho)^{-1} g_{\operatorname{Euc}} \;+\; \rho^{-2}\,\partial\rho\,\overline{\partial}\rho,
\]  
where $g_{\operatorname{Euc}}$ denotes the Euclidean metric.
\end{lemma}

\begin{proof}
Set $ds^2 = (-\rho)^{-1}g_{\operatorname{Euc}} + \rho^{-2}\partial\rho\,\overline{\partial}\rho$. Since $ds^2_z(X,X)$ and $g_{\operatorname{FS},z}(X,X)$ are homogeneous in $X$, we can assume $|X|=1$ throughout the proof. On compact subsets the equivalence is trivial thus it is enough to consider points near the boundary, and consider the same two cases as in the proof of Lemma \ref{theta_bounded}.

\medskip

 \emph{Case N} ($\nu^2 \ge C\tau$). By \eqref{N:lowerbdd:feffsgo_metric}, we have 
\[g_{\operatorname{FS}}(X,X)\ge \frac{n}{2}\left(1-\frac{M}{C}\right)\nu(z)^2.\] The upper bound of $g_{\operatorname{FS}}$ is
\[g_{\operatorname{FS}}(X,X)\leq n\left(1+\frac{M}{C}\right)\nu^2+o(\nu)^2, \]
which follows from $\nu^2 \ge C\tau$ and \eqref{N:Leviform_h}. 

Thus, we get $g_{\operatorname{FS}} \approx \nu^2$, while $ds^2 = \tau + \nu^2 \approx \mathfrak \nu^2$. Therefore, $g_{\operatorname{FS}}\approx ds^2$ near $\pa\Om$. 

\medskip

\emph{Case T} ($\nu^2 < C\tau$). From \eqref{T:lowerbdd:feffsgo_metric}, we have 
\[g_{FS}\geq \frac{nc_0}{32}\tau.\]
Using \eqref{T:Leviform:h} and $\nu^2 < C\tau$, we obtain the upper bound of $g_{\operatorname{FS}}$ as 
\[g_{\operatorname{FS}}(X,X)\leq n(C+M)\tau+o(\tau)=O(\tau).\]

This yields $g_{\operatorname{FS}} \approx \tau$. Since $ds^2=\tau+\nu^2 \approx \tau$, we finally obtain $g_{\operatorname{FS}}\approx ds^2$ near $\pa\Om$.  
\end{proof}

Now take $D=\Omega$, $\phi=\rho$ (so that $-\phi = -\rho >0$ in $\Omega$), $ds_U^2 = g_{\operatorname{Euc}}$, $a=1$, $b=2$. Strict pseudoconvexity guarantees that every boundary point is regular and non‑degenerate. The hypothesis of Theorem \ref{ohsawa} is satisfied, and we conclude
\[
\dim H_2^{p,q}(\Omega) = \infty \quad \text{for all }\; p+q=n.
\]
This completes the proof of Theorem \ref{l2-co}.

\section{Characterization of the complex unit ball}
In this section, we prove Theorem~\ref{main:rigidity}. 
A key ingredient in the proof of Part~(ii) is the $C^\infty$-bounded geometry of the Fefferman--Szeg\H{o} metric established in Part~(i), which we prove first.

\subsection{Bounded geometry}
The proof of Theorem \ref{main:rigidity} (i) is inspired by \cite{za21}. We begin by recalling the notions of quasi-bounded geometry and bounded geometry.

\begin{definition}
    Let $(M,g)$ be a complete K\"ahler manifold. For $x\in M$, the injectivity radius at $x$ is the largest $r>0$ such that the exponential map $\exp_{x}: \mathbb{B}_{r}(\subset T_{x}M)\to M$ is a diffeomorphism onto its image. The injectivity
radius of $M$ is the infimum of the injectivity radius over all points of $M$.
\end{definition}

\begin{definition}\label{def gem}
    Let $(M,g)$ be a complete K\"ahler manifold and let $s$ be a nonnegative integer. We say  that $(M,g)$ has $C^{s}$-quasi-bounded geometry if for every integer $0\leq p\le s$, there exists a constant $C_{p} > 0$ such that
    \[\sup_{M}\left\|\nabla^p \operatorname{Rm}\right\|_{g}\le C_{p},\]
    where $\operatorname{Rm}$ denotes the Riemann curvature tensor of $g$. 
\end{definition}
 If, in addition, the injectivity radius of $(M, g)$ is positive, then  $(M, g)$ is said to have $C^{s}$-bounded geometry.
\begin{definition}\label{defn:BBG}
We say that a domain $\Omega \subset \mathbb{C}^{n}$ has bounded intrinsic geometry if it admits a complete K\"ahler metric $g$ such that the following hold:
\begin{enumerate}[label=(b.\arabic*)]
\item\label{item:bd_sec} $g$ has bounded sectional curvatures and positive injectivity radius;

\item\label{item:SBG} there exists a $C^2$ function $f : \Omega \rightarrow \mathbb{R}$ such that  the Levi form of $f$ is uniformly bi-Lipschitz to  $g$ and $\norm{\partial^{1,0} f}_{g}$ is bounded on $\Omega$. Here, $\partial^{1,0} f$ means $\partial^{1,0} f=\partial_i fdz^i$.
\end{enumerate}
\end{definition}
\begin{remark}\label{b2_feffsgo}
   Property \ref{item:SBG} for the Fefferman–Szeg\H o metric follows from Lemma \ref{theta_bounded}. Moreover, every bounded strongly pseudoconvex domain with $C^2$-smooth boundary has bounded intrinsic geometry; see \cite{za21}.
\end{remark}
We shall use the following two results from the work of Zimmer \cite{za21}.
\begin{theorem}\cite[Theorem 1.2]{za21}\label{thm:bergman_reduction}
Let $\Omega \subset \mathbb{C}^{n}$ be a domain. Then $\Omega$ has bounded intrinsic geometry if and only if the Bergman metric $g_{\operatorname{B}}^\Omega$ satisfies Definition~\ref{defn:BBG}.
Moreover, for every $m\geq 0$,
\begin{align*}
\sup_{z \in \Omega} \norm{\nabla^m \operatorname{Rm}}_{g_{\operatorname{B}}^\Omega} < \infty
\end{align*}
 where $\operatorname{Rm}$ is the Riemann curvature tensor of $g_{\operatorname{B}}^\Om$.
\end{theorem}

\begin{theorem}\cite[Theorem 5.1 (1)]{za21}\label{thm:charts} 
Let $\Omega \subset \mathbb{C}^n$ be a domain, and let $g$ be a complete K\"ahler metric on $\Omega$. If $g$ satisfies Property~\ref{item:bd_sec}, then for every $\zeta \in \Omega$, there exists a holomorphic embedding $\Phi_\zeta : \mathbb{B}^n \rightarrow \Omega$ with  $\Phi_\zeta(0) = \zeta$, satisfying $\Phi_\zeta^* g \approx g_{\operatorname{Euc}}.$
\end{theorem}

\begin{remark}\label{ree}
In the above result, we can choose $\Phi_\zeta$ which extends smoothly to $\pa \mathbb B^n$. Throughout the remainder of the section, we fix such a choice of maps $\Phi_\zeta$.
\end{remark}

\begin{proof}[Proof of Theorem \ref{main:rigidity} (i)]
   Let $\Omega\subset\mathbb{C}^{n}$ be a strongly pseudoconvex bounded domain with $C^\infty$-smooth boundary. It is known that $\Omega$ has bounded intrinsic geometry and therefore, there exists a complete K\"ahler metric $g$ that satisfies Definition \ref{defn:BBG}. We can take $g$ to be $g_{\operatorname{B}}^\Om$  \ref{defn:BBG} because of Theorem \ref{thm:bergman_reduction}. Hence by Theorem~\ref{thm:charts}, there exists a holomorphic embedding $\Phi_{\zeta}: \mathbb{B}^n\to \Omega$ that extends $C^\infty$-smoothly to the boundary which satisfies 
   \begin{equation}
      \Phi_\zeta(0)=\zeta\quad\text{ and }\quad \Phi_\zeta^* g_{\operatorname{B}}^\Om \approx g_{\operatorname{Euc}}.
   \end{equation}
   
   Using these functions, we introduce the following “local Fefferman-Szeg\H o kernels" as follows. For every $\zeta\in\Omega$, define $S^{\text{loc}}_{\zeta}: ~\mathbb{B}^n\times\mathbb{B}^n\to\mathbb{C}$ as
\begin{align*}
    S^{\text{loc}}_{\zeta}(z,w)=S_{\Omega}(\Phi_{\zeta}(z),\Phi_{\zeta}(w))\left(\det J_{\mathbb C}\Phi_{\zeta}(z)\right)^{\frac{n}{n+1}}\left(\overline{\det J_{\mathbb C}\Phi_{\zeta}(w)}\right)^{\frac{n}{n+1}}.
\end{align*}

We first establish estimates for the function $S^{\text{loc}}_\zeta$ and its derivatives. 
\begin{lemma}\label{local_kernel} We have the following
\begin{enumerate}
\item There exists a constant $C>1$, independent of $\zeta $, such that 
\[C^{-1}\leq S^{\text{loc}}_\zeta(z,z)\leq C\]
$\text{ for all }\zeta\in\Om\text{ and } z\in\mathbb B^n$.
\item  If $\delta \in (0,1)$, then for all multi-indices $a,b$ there exists $C_{a,b}=C_{a,b}(\delta) > 0$ such that
\begin{align*}
\sup_{\zeta\in\Om;~z,w\in\mathbb B^n_\delta}\abs{\frac{\partial^{\abs{a}+\abs{b}}S^{\text{loc}}_\zeta}{\partial z^{a}\partial \bar{w}^b}(z,w)} \leq C_{a,b}.
\end{align*}

\end{enumerate}
\end{lemma}


To prove this lemma, we need the following result.
\begin{lemma}\label{lem:unif_local_bd_1} There exists a constant $C>1$ such that $C^{-1}S_{\Omega}(\zeta)\leq S_{\Phi_\zeta(\mathbb{B}^n)}(\zeta)\leq CS_\Om(\zeta)$.
\end{lemma}

\begin{proof}
Let $K_D(z)$ denotes the diagonal values of the Bergman kernel $K_D(z,w)$ for domain $D\subset\mathbb C^n$. From \cite[Lemma 9.2]{za21}, there exists a constant $c_{1}>1$ such that for all $\zeta\in \Omega$,
\begin{equation}\label{first}
    K_{\Omega}(\zeta) \leq K_{\Phi_\zeta(\mathbb{B}^n)}(\zeta) \leq c_1K_{\Omega}(\zeta).
\end{equation}
 Then consider the $SK_\Om$ function:
\begin{equation}\label{second}
    SK_{\Omega}(z,w)=\frac{S_{\Omega}^{n+1}(z,w)}{K_{\Omega}^{n}(z,w)}\in C(\overline{\Om}),
\end{equation}
see \cite[Theorem 2]{bl14} and \cite{bl16}. Combining (\ref{first}) and (\ref{second}), we have 
\begin{equation}
    S_{\Omega}^{\frac{n+1}{n}}( \zeta)SK_{\Omega}^{-\frac{1}{n}}( \zeta,\zeta) \leq S_{\Phi_\zeta(\mathbb{B}^n)}^{\frac{n+1}{n}}(\zeta)SK_{\Phi_\zeta(\mathbb{B}^n)}^{-\frac{1}{n}}(\zeta,\zeta) \leq c_1S_{\Omega}^{\frac{n+1}{n}}( \zeta)SK_{\Omega}^{-\frac{1}{n}}(\zeta,\zeta).
\end{equation}
Since $SK$ is invariant (see \cite[Theorem 1]{bl14}), and $SK_{\mathbb{B}^n}$ is identically constant (depending only on $n$) see \cite[p. 10]{bl14}, 
\begin{equation}
    S_{\Omega}^{\frac{n+1}{n}}(\zeta)\leq c_{2}S_{\Phi_\zeta(\mathbb{B}^n)}^{\frac{n+1}{n}}(\zeta)SK_{\Omega}^{\frac{1}{n}}( \zeta,\zeta)\leq c_1S_{\Omega}^{\frac{n+1}{n}}(\zeta).
\end{equation}
 It implies that
$$ \frac{S_{\Omega}(\zeta)}{c_{2}SK_{\Omega}^{\frac{1}{n+1}}(\zeta,\zeta)}\leq S_{\Phi_\zeta(\mathbb{B}^n)}(\zeta)\leq \frac{c_1S_{\Omega}(\zeta)}{c_{2}SK_{\Omega}^{\frac{1}{n+1}}(\zeta,\zeta)}.$$
By $SK_\Om\in C(\overline{\Omega})$ and $SK_\Om(\zeta,\zeta)>0$ on $\overline\Omega$, there exists a constant $C>1$ such that
$$C^{-1}S_{\Omega}(\zeta)\leq S_{\Phi_\zeta(\mathbb{B}^n)}(\zeta)\leq CS_{\Omega}(\zeta),$$
which completes the proof.
\end{proof}
We now prove Lemma \ref{local_kernel}.  
In the following proof, for each $z\in\Omega$, we identify the Bergman metric
$g_{\operatorname{B},z}^{\Omega}$ with the positive definite Hermitian matrix $\left[
g_{\operatorname{B},z}^{\Omega}
\!\left(
\partial_{i},
\partial_{\bar{j}}
\right)
\right]_{i,j=1}^n .$
\begin{proof}[Proof of Lemma \ref{local_kernel}.]
For fixed $\zeta\in\Om$, we have
\begin{align*}
S_{\zeta}^{\mathrm{loc}}(w,w)
&=
S_{\Omega}\bigl(\Phi_\zeta(w)\bigr)
\abs{\det J_\mathbb C\Phi_\zeta(w)}^{\frac{2n}{n+1}} \\
&=
S_{\Omega}\bigl(\Phi_\zeta(w)\bigr)
\left(
\frac{
\det\bigl(\Phi_\zeta^*g_{\mathrm B}^{\Omega}\bigr)_w
}{
\det g_{\mathrm B,\Phi_\zeta(w)}^{\Omega}
}
\right)^{\frac{n}{n+1}} 
\approx
\frac{
S_{\Omega}\bigl(\Phi_\zeta(w)\bigr)
}{
\bigl(\det g_{\mathrm B,\Phi_\zeta(w)}^{\Omega}\bigr)^{\frac{n}{n+1}}
},
\end{align*}
where the last comparability follows from Theorem~\ref{thm:bergman_reduction}. 

Hence it suffices to show that there exists a constant \(C>1\) such that \[C^{-1}\leq \frac{S_\Om(z)}{(\det g_{\operatorname{B},z}^\Om)^{\frac{n}{n+1}}}\leq C\] for all \(z\in\Omega\).
 Since $\Phi_\zeta:~\mathbb{B}^n\to \Phi_\zeta(\mathbb{B}^n)$ is biholomorphic and $\Phi_\zeta(0)=\zeta$, using the transformation formula, we have
\begin{align}
\frac{(n-1)!}{\pi^n}=S_{\mathbb{B}^n}(0)=S_{\Phi_\zeta(\mathbb{B}^n)}(\zeta) \abs{\det J_{\mathbb C}\Phi_\zeta(0)}^{\frac{2n}{n+1}}
=S_{\Phi_\zeta(\mathbb{B}^n)}(\zeta) \left(\frac{\det \big(\Phi_\zeta^*g_{\operatorname{B}}^\Om\big)_0 }{ \det g_{\operatorname{B},\zeta}^\Om}\right)^{\frac{n}{n+1}}.
\end{align}
The first part now follows from
Lemma~\ref{lem:unif_local_bd_1} and
Theorem~\ref{thm:bergman_reduction}.

Next, for all $z,w\in\mathbb B^n$, 
\begin{multline*}
\abs{S_\zeta^{\operatorname{loc}} (z,w)}
=\abs{S_{\Omega}\big(\Phi_{\zeta}(z),\Phi_{\zeta}(w)\big)\left(\det J_{\mathbb C}\Phi_{\zeta}(z)\right)^{\frac{n}{n+1}}\left(\overline{\det J_{\mathbb C}\Phi_{\zeta}(w)}\right)^{\frac{n}{n+1}}}\\
\le\sqrt{S_{\Omega}\big(\Phi_{\zeta}(z)\big)\abs{\det J_{\mathbb C}\Phi_{\zeta}(z)}^{\frac{2n}{n+1}}}\sqrt{S_{\Omega}\big(\Phi_{\zeta}(w)\big)\abs{\det J_{\mathbb C}\Phi_{\zeta}(w)}^{\frac{2n}{n+1}}}\\
\leq\sqrt{S_\zeta^{\operatorname{loc}}(z,z)}\sqrt{S_\zeta^{\operatorname{loc}}(w,w)} \leq C,
\end{multline*}
 where the last inequalty follows from the first part. Since $S_\zeta^{\operatorname{loc}}$ is holomorphic in the first variable and anti-holomorphic in the second variable, the derivative estimates follow
from Cauchy's integral formulas. This  completed the second part of the proof.
\end{proof}

We now continue the proof of Theorem \ref{main:rigidity}. Recall that $g_{\operatorname{FS}}^\Om$ is complete (see \cite[Theorem 6]{bl14}) and it satisfies Property \ref{item:SBG}) of Definition \ref{defn:BBG} (see Remark \ref{b2_feffsgo}).
From \cite{bl14}, it follows that 
\begin{equation}
    \frac{1}{C}g_{\operatorname{B}}^\Om\le g_{\operatorname{FS}}^\Omega\le C g_{\operatorname{B}}^\Om
\end{equation}
which $C$ depends only on $\Omega$. Using the fact $\Phi_{\zeta}^* g_{\operatorname{B}}^\Om \approx g_{\operatorname{{Euc}}}$, we know 
\begin{equation}\label{pull back}
     \Phi_\zeta^* g_{\operatorname{FS}}^\Omega\approx g_{\operatorname{Euc}}. 
\end{equation}

Furthermore, since $w \longmapsto\det J_{\mathbb C}\Phi_\zeta(w)^{\frac{n}{n+1}}$ is holomorphic on $\mathbb{B}^n$, so we have
\begin{align*}
\Phi_\zeta^* g_{\operatorname{FS}}^\Omega(w)=g_{\operatorname{FS}}^\Om(\Phi_\zeta(w)) 
&=\partial_{i\overline{j}} \log S_\Omega(\Phi_\zeta(w))dw^i d\bar{w}^j=\partial_{i\overline{j}} \log S_\zeta^{\operatorname{loc}}(w,w)dw^i d\bar{w}^j.
\end{align*}
In local coordinates, the curvature tensor of the K\"ahler metric $\mathfrak g := \Phi_{\zeta}^{*} g_{\operatorname{FS}}^\Omega$ is given by
\begin{equation}\label{equ1}
    R(\mathfrak g)_{i\bar{j}k\bar{\ell}} = \partial_{i\bar{j} k \bar{\ell}}\big(  \log S_\zeta^{\operatorname{loc}}\big) - \mathfrak g^{p\bar{q}} \Big(\partial_{ik\bar{p}}  \big(\log S_\zeta^{\operatorname{loc}}\big) \partial_{\bar{j}\bar{\ell}q} \big( \log S_\zeta^{\operatorname{loc}}\big)\Big).
\end{equation}
The components of the covariant derivatives $\nabla^{p} \operatorname{Rm}(\mathfrak g)$ can be expressed in terms of higher-order derivatives of $\log S_\zeta^{\operatorname{loc}}$, the components of $\mathfrak g$, and its inverse. Then by Lemma \ref{local_kernel} together with the fact that $\mathfrak g\approx g_{\operatorname{Euc}}$, we obtain
\begin{equation}
    \| \nabla^{p} \operatorname{Rm}(\mathfrak g) \|_{\mathfrak g}(0) \leq C_{p}.
\end{equation}
Since $\Phi_{\zeta}^{*} g_{\operatorname{FS}} ^\Omega= \mathfrak g$, then $\Phi_{\zeta}^{*} \left(\nabla^{p} \operatorname{Rm}\big(g_{\operatorname{FS}}^\Omega\big)\right) = \nabla^{p} \operatorname{Rm}(\mathfrak g)$. It implies that
\begin{equation}
   \| \nabla^{p} \operatorname{Rm}(\mathfrak g) \|_{\mathfrak g}(0) = \| \nabla^{p} \operatorname{Rm}\big(g_{\operatorname{FS}}^\Omega\big) \|_{g_{\operatorname{FS}}^\Omega} (\zeta). 
\end{equation}
Thus for any $p\ge 0$ and $\zeta\in\Om$ we have proved
\begin{equation}\label{qqu}
    \sup_{z \in \Omega} \| \nabla^{p} \operatorname{Rm}\big(g_{\operatorname{FS}}^\Om\big) \|_{g_{\operatorname{FS}}^\Omega} \leq C_{p} < +\infty,
\end{equation}

Finally, $g_{\operatorname{FS}}^\Om$ has bounded sectional curvature by (\ref{qqu}) (also see Theorem 1.1 of \cite{b26}). Then [\cite{LSY05}, Proposition 2.1] with (\ref{pull back}) implies that $g_{\operatorname{FS}}^\Om$ has positive injectivity radius. Hence, by the Definition \ref{def gem}, the Fefferman-Szeg\H o metric on the strongly pseudoconvex domian with smooth boundary has $C^\infty$-bounded geometry. The first part of Theorem \ref{main:rigidity} is proved.
\end{proof}

\subsection{Proof of Theorem \ref{main:rigidity} (ii)}
To prove Theorem \ref{main:rigidity} (ii) and (iii), we begin by recalling two key prior results: the boundary behaviour of the Ricci and scalar curvatures of the Fefferman-Szeg\H{o} metric in \cite{b26}, and the analogous of Cheng's conjecture for the same metric in \cite{y25} as follows.
\begin{proposition}\label{equa f}
    Let $\Omega\subset\mathbb{C}^{n}$ be a a bounded pseudoconvex domain and let  $p\in\partial\Omega$ be a $C^{\infty}$-smooth strongly pseudoconvex point. Then 
         $$\lim_{z\to p}\left|\operatorname{Ric}(g_{\operatorname{FS}})+\frac{n+1}{n}g_{\operatorname{FS}}\right|(z)=0, \quad \text{and} \quad 
        \lim_{z\to p}\mathrm{Sca}(g_{\operatorname{FS}})=-(n+1).$$
\end{proposition}
\begin{remark}\label{neg_pin_feff-sgo-metric}
    The Ricci curvature of the Fefferman-Szeg\H{o} metric $g_{\operatorname{FS}}$ is negatively pinched near the boundary, i.e., there exists a compact subset $K \subset \subset \Omega$ and some constants $C_2 \geq C_1 >0$ such that $ -C_2 g_{\operatorname{FS}}\leq \operatorname{Ric}(g_{\operatorname{FS}})\leq -C_1 g_{\operatorname{FS}}<0$ outside $K$ .
\end{remark}
\begin{theorem}\label{Yuan}
    The Fefferman-Szeg\H o metric metric of a bounded strongly pseudoconvex domain $\Omega$ with $C^\infty$-boundary is K\"ahler-Einstein if and only if the domain is biholomorphic to the ball.
\end{theorem}
Next, we recall a general result which characterises K\"ahler-Einstein metric among K\"ahler-Ricci solitons given by Sha \cite{s26a}.
\begin{theorem}\label{MT1}
Let $\Omega \subset \mathbb{C}^n$ be a bounded pseudoconvex domain with $C^2$ boundary, and let $g$ be a complete K\"ahler metric on $\Omega$ of $C^1$-bounded geometry. Suppose there exists a compact subset $K \subset \subset \Omega$ such that $0> -C_1 g\geq \operatorname{Ric}(g)\geq -C_2 g$ outside $K$ for some constants $C_2 \geq C_1 >0$. If $g$ is a K\"ahler-Ricci soliton with a real holomorphic vector field $X$ such that the dual 1-form of $X$ with respect to $g$ is closed, then $g$ is K\"ahler-Einstein.
\end{theorem}

From Theorem \ref{main:rigidity} (i) and Remark \ref{neg_pin_feff-sgo-metric}, the Fefferman-Szeg\H o metric $g_{\operatorname{FS}}$ satisfies the conditions of Theorem \ref{MT1}.

\begin{proof}[Proof of Theorem \ref{main:rigidity} (ii)] 
The 'if' part: if $g_{\operatorname{FS}}$ is a gradient K\"ahler-Ricci soliton, then there exists a real holomorphic vector field such that 
\begin{equation}
    \operatorname{Ric}(g_{\operatorname{FS}})+\mathcal{L}_{X}g_{\operatorname{FS}}=-\frac{n+1}{n}g_{\operatorname{FS}}.
\end{equation}
Moreover, there exists a real smooth function $f$ such that $X=\frac{1}{2}\nabla f$, implying the dual $1$-form $X^{\mathfrak{b}}$ of $X$ is closed. Indeed, $X^{\mathfrak{b}}=\frac{1}{2}df$, then $dX^{\mathfrak{b}}=0$. Combining the above argument and Theorem \ref{MT1}, the gradient K\"ahler-Ricci soliton is K\"ahler-Einstein, then $\Om$ is biholomorphic to the ball by Theorem \ref{Yuan}.
    
The 'only if' part: Because $\Omega$ is biholomorphic to the ball and the Ricci curvature is invariant under biholomorphic map, so
    $$\operatorname{Ric}(g_{\operatorname{FS}})=-\frac{n+1}{n}g_{\operatorname{FS}},$$
    which implies that $g_{\operatorname{FS}}$ is trivial K\"ahler–Ricci soliton. We can choose $X=0$, then
    $$\operatorname{Ric}(g_{\operatorname{FS}})+\mathcal{L}_Xg_{\operatorname{FS}}=-\frac{n+1}{n}g_{\operatorname{FS}},$$
    where $X$ is a trivial gradient vector field. In this case, $g_{\operatorname{FS}}$ is a gradient K\"ahler–Ricci soliton. Then the proof of the second part of Theorem \ref{main:rigidity} is completed.
\end{proof}

\begin{remark}
    This indicates that there exist no non-trivial gradient K\" ahler-Ricci soliton of the Fefferman-Szeg\H o metric on strongly pseudoconvex domain with $C^\infty$-smooth boundary. 
\end{remark}
\subsection{Proof of Theorem \ref{main:rigidity} (iii)}
\begin{proof}
 Using the definiton of $\mathfrak{R}(z)$, we get
  \[\log \det G(z)=\frac{n+1}{n}\log S(z)+\log \mathfrak{R}(z).\]
Then, taking $\partial_{i\bar j}$ on both sides,
it follows that
\begin{align*}
\operatorname{Ric}_{i\bar j}+\pa_{i\ov j}\log \mathfrak{R}(z)=-\frac{n+1}{n}g_{i\bar j}.
\end{align*}
Taking trace with respect to $g_{\operatorname{FS}}$, we have
\begin{align*}
\operatorname{tr}_{g_{\operatorname{FS}}}\operatorname{Ric}(g_{\operatorname{FS}})+\operatorname{tr}_{g_{\operatorname{FS}}}\left(\pa_{i\ov j}\log \mathfrak{R}(z)\right)=-\frac{n+1}{n}\operatorname{tr}_{g_{\operatorname{FS}}}g_{\operatorname{FS}}.
\end{align*}
Since $g_{\operatorname{FS}}$ has constant scalar curvature, Proposition \ref{equa f} implies 
$$-(n+1)+\Delta_{g_{\operatorname{FS}}} \log \mathfrak{R}(z)=-(n+1),$$
then,
$$ \Delta_{g_{\operatorname{FS}}} \log \mathfrak{R} \equiv 0 \quad \text{in} \quad \Omega.$$
This implies $\log \mathfrak{R}$ is harmonic with respect to $g_{\operatorname{FS}}$ with the constant boundary value $((n-1)!)^{-\frac{n+1}{n}}n^{n}\pi^{n+1}$; see \cite[Theorem 1.1 (b)]{b26}). By the maximum principle for such harmonic functions, it follows that 
\[
    \log \mathfrak{R} \equiv -\frac{n+1}{n}\log (n-1)!+n\log n+(n+1)\log\pi.
\]
Then using \cite[Corollary 4.8]{y25} we have $g_{\operatorname{FS}}$ is K\"ahler-Einstein. The moreover part follows immediately from Theorem \ref{Yuan}.
\end{proof}
\section{Ramadanov-type conjecture for the Fefferman-Szeg\H o kernel}
In this section, we prove Theorem~\ref{gen_rama_conj}. Throughout this section $\Omega$ denotes a strongly pseudoconvex domain with $C^\infty$-smooth boundary. In view of Theorem~\ref{gen_rama_conj}, it suffices to work with a Fefferman defining function $\rho$ for $\Om=\{z\in\mathbb C^n:\rho(z)>0\}$. We begin by recalling the asymptotic expansion of the Fefferman-Szeg\H{o} kernel (see \cite{fc74,bs76}):
    \[S(z)=\frac{\phi(z)}{\rho(z)^{n}}+\psi(z)\log\rho(z)=\frac{h(z)}{\rho(z)^n},\]
where $\phi,\psi\in C^{\infty}(\overline{\Omega})$ and $\phi|_{\pa\Om}=(n-1)!/\pi^n$. 

The following result of Hirachi--Komatsu--Nakazawa \cite[Proposition 1 and Remark 2]{hkn93} gives the asymptotic behaviour of $\phi$ and $\psi$.
    \begin{itemize}
        \item If $n=2$, then
        $$\phi=\pi^{-2}+O(\rho^{2}),\quad \psi=k_{1}\eta_{1}\rho+k_{2}|A_{2\bar 4}^{0}|\rho^{2}+O(\rho^{3}),$$
      where $k_1$, $k_2$ are universal constants independent of the domain $\Omega$. Moreover, $\partial\Omega$ is locally spherical if and only if $\psi=O(\rho^{3})$.
        \item  If $n\ge 3$, then
        \begin{equation}\label{asym:phi_n_ge_3}
            \phi=\frac{(n-1)!}{\pi^n}+b_{n}|A_{2\bar2}^{0}|^{2}\rho^{2}+O(\rho^{3}),
        \end{equation}
        where $b_n$ is a universal constant only depending on $n$. 
    \end{itemize}
\begin{remark}
    For $n\geq 3$, the boundary $\partial\Omega$ is locally spherical if and only if
$\phi=\frac{(n-1)!}{\pi^n}+O(\rho^{3}).$
\end{remark}
For the remainder of this section, we assume $n\ge4$.
\begin{lemma}
 There exist functions $q,\ti q\in C^\infty(\overline\Omega)$ such that $h$ can be expressed as 
   \[h=\left(\frac{n^{n}}{C_{n}}\right)^{\frac{n}{n+1}}(1+q\rho^{2}+\ti q\rho^{3}+o(\rho^{3}))\] where $C_{n}=\left((n-1)!\right)^{-\frac{n+1}{n}}n^{n}\pi^{n+1}$.
\end{lemma}

\begin{proof}
     From the definition of $h$, we have $h=\phi+o(\rho^{n-1})$. For $n\ge 4$, this implies $h=\phi+o(\rho^{3})$. Using $\phi\in C^\infty(\ov\Om)$, $\phi|_{\pa\Om}=(\frac{n^{n}}{C_{n}})^{\frac{n}{n+1}}$ and \eqref{asym:phi_n_ge_3}, we have 
$$h=\left(\frac{n^{n}}{C_{n}}\right)^{\frac{n}{n+1}}(1+q\rho^{2}+\ti q\rho^{3}+o(\rho^{3})).$$
where $q,\tilde q\in C^\infty(\overline{\Omega})$. 
\end{proof}
Set $Q:=(\frac{C_{n}}{n^{n}})^{\frac{n}{n+1}}h$, then
\begin{equation}\label{Q}
    Q=1+q\rho^{2}+\ti q\rho^{3}+o(\rho^{3}).
\end{equation} Next, we define 
\[\Phi=\left(\frac{C_{n}}{n^n}S^{\frac{n+1}{n}}\right)^{-\frac{1}{n+1}}=\rho Q^{-\frac{1}{n}}.\] 
Since $Q\neq0$ in $\Omega$, $\Phi$ is well-defined. 

The proof is based on deriving an asymptotic expansion for \(J[\Phi]\)
near \(\partial\Omega\). We first express \(J[\Phi]\) in terms of the
Fefferman defining function \(\rho\) and the coefficient \(q\) arising in
the asymptotic expansion of the Fefferman--Szeg\H{o} kernel. We then use
the identity $J[\Phi]=\mathfrak R/C_n,$ which follows from the definition of the Fefferman--Szeg\H{o} invariant
function, to obtain the asymptotic behaviour of \(\mathfrak R\). 

We begin by recording the derivatives of $\Phi$ in terms of $Q$.
\begin{lemma}\label{Phi-derivatives}
    For $\al,\be=1,\cdots,n$,  we have 
\begin{align*}
       \Phi_\al&=\Phi(\rho^{-1}\rho_{\alpha}-\frac{1}{n}Q^{-1}Q_{\alpha})\\
        \Phi_{\ov\be}&=\Phi(\rho^{-1}\rho_{\ov\be}-\frac{1}{n}Q^{-1}Q_{\ov\be})\\
         \Phi_{\al\ov\be}&=Q^{-\frac{1}{n}}\left(-\frac{1}{n} Q^{-1}\left(Q_{\alpha} \rho_{\bar{\beta}}+Q_{\bar{\beta}} \rho_{\alpha}\right)+\frac{n+1}{n^{2}} Q^{-2} Q_{\alpha} Q_{\bar{\beta}} \rho+\rho_{\alpha \bar{\beta}}-\frac{1}{n} Q^{-1} Q_{\alpha \bar{\beta}} \rho\right).
   \end{align*}
\end{lemma}
\begin{proof}
    The proof of this lemma follows from the direct calculations.
\end{proof}
We now compute the Monge–Amp\'ere operator of $\Phi$.
\begin{lemma}\label{J-Phi}
    We have $J[\Phi]=\frac{(-1)^{n}}{Q^\frac{n+1}{n}}\det\begin{bmatrix}
           \rho & \rho_{\overline{\beta}}\\
           \rho_{\alpha} & \rho_{\alpha\bar\beta}-\frac{1}{n}\rho(\log Q)_{\al\ov\be}
       \end{bmatrix}.$
\end{lemma}
\begin{proof}
By Lemma \ref{Phi-derivatives},
$$J[\Phi]=(-1)^{n}\det\begin{bmatrix}
 Q^{-\frac{1}{n}} \rho & Q^{-\frac{1}{n}}\left(\rho_{\overline{\beta}}-\frac{1}{n} Q^{-1} Q_{\overline{\beta}} \rho\right) \\
Q^{-\frac{1}{n}}\left(\rho_{\alpha}-\frac{1}{n} Q^{-1} Q_{\alpha} \rho\right) & \Phi_{\alpha \overline{\beta}}
\end{bmatrix}.
$$
 Factoring $Q^{-\frac{1}{n}}$ from each of the $n+1$ rows gives 
$$J[\Phi]=\frac{(-1)^{n}}{Q^\frac{n+1}{n}}\det\begin{bmatrix}
  \rho & \rho_{\overline{\beta}}-\frac{1}{n} Q^{-1} Q_{\overline{\beta}} \rho \\
\rho_{\alpha}-\frac{1}{n} Q^{-1} Q_{\alpha} \rho & \Phi^{(1)}_{\alpha \overline{\beta}}
\end{bmatrix}
$$
where 
\[
\Phi^{(1)}_{\alpha \overline{\beta}}=-\frac{1}{n} Q^{-1}\left(Q_{\alpha} \rho_{\bar{\beta}}+Q_{\bar{\beta}} \rho_{\alpha}\right)+\frac{n+1}{n^{2}} Q^{-2} Q_{\alpha} Q_{\bar{\beta}} \rho+\rho_{\alpha \bar{\beta}}-\frac{1}{n} Q^{-1} Q_{\alpha \bar{\beta}} \rho.
\]
Next, add $\frac{1}{n} Q^{-1} Q_{\bar{\beta}}$ times the first column to the $(\beta+1)$-th column. Then 
\[
J[\Phi]=\frac{(-1)^n}{Q^\frac{n+1}{n}} \begin{bmatrix}
    \rho & \rho_{\overline{\beta}} \\
\rho_{\alpha}-\frac{1}{n} Q^{-1} Q_{\alpha} \rho & \Phi^{(2)}_{\alpha \overline{\beta}}
\end{bmatrix},
\]
where
\[
\Phi^{(2)}_{\alpha \overline{\beta}}=-\frac{1}{n} Q^{-1} Q_{\alpha} \rho_{\overline{\beta}}+\frac{1}{n} Q^{-2} Q_{\alpha} Q_{\overline{\beta}} \rho+\rho_{\alpha \overline{\beta}}-\frac{1}{n} Q^{-1} Q_{\alpha \overline{\beta}} \rho.
\]
Finally, adding $\frac{1}{n} Q^{-1} Q_{\alpha}$ times the first row to the $(\alpha+1)$-th row completes the proof.
\end{proof}
We now derive the asymptotic expansion of $J[\Phi]$.
\begin{lemma}\label{approx:J_exp}
We have
    \begin{equation*}
        J[\Phi]=\left(1-\frac{n+1}{n}q \rho^2\right)\left(1+\frac{2}{n} q \rho^2\right)\left(1-\frac{2}{n} q \rho^2\right)^{n-1} J[\rho]+o\left(\rho^2\right).
    \end{equation*}
\end{lemma}
\begin{proof}
    Using the Taylor expansion $(1+x)^{-\frac{n+1}{n}}=1-\frac{n+1}{n}x+O(x^2)$ together with \eqref{Q}, we obtain
\[
Q^{-\frac{n+1}{n}}=1-\frac{n+1}{n}q \rho^2+o\left(\rho^2\right).
\]

Next, using $\log (1+x)=x+O(x^{2})$, we obtain \[
\log Q=q \rho^2+\ti q \rho^3+o\left(\rho^3\right).
\]
Differentiating gives
\[
(\log Q)_{\alpha}=q_{\alpha} \rho^2+2 q \rho \rho_{\alpha}+3 \ti q \rho^2 \rho_{\alpha}+o\left(\rho^2\right)
\]
\[
\begin{gathered}
(\log Q)_{\alpha \overline{\beta}}
=2 q \rho_{\alpha} \rho_{\overline{\beta}}+2\left(q_{\alpha} \rho_{\overline{\beta}}+q_{\overline{\beta}} \rho_{\alpha}+q\rho_{\alpha \overline{\beta}}+3 \ti q \rho_{\alpha} \rho_{\overline{\beta}}\right) \rho+o(\rho).
\end{gathered}
\]
Substituting this into Lemma~\ref{J-Phi}, we obtain

\[
J[\Phi]=(-1)^n\left(1-\frac{n+1}{n}q \rho^2+o\left(\rho^2\right)\right) \times
\]
\[
\det\begin{bmatrix}
\rho & \rho_{\overline{\beta}} \\
\rho_{\alpha} & \rho_{\alpha \overline{\beta}}-\frac{2}{n}\left(q \rho_{\alpha} \rho_{\overline{\beta}} \rho+\left(q_{\alpha} \rho_{\overline{\beta}}+q_{\overline{\beta}} \rho_{\alpha}+q \rho_{\alpha \overline{\beta}}+3 \ti q \rho_{\alpha} \rho_{\overline{\beta}}\right) \rho^2\right)+o\left(\rho^2\right)
\end{bmatrix}.
\]
Next, add $\frac{2}{n}(q \rho_{\alpha} \rho+(q_{\alpha}+3 \ti q\rho_{\alpha}) \rho^2)$ times the first row to the $(\alpha+1)$-th row, which gives 
$$J[\Phi]=(-1)^{n}\left(1-\frac{n+1}{n}q \rho^2+o\left(\rho^2\right)\right)\times$$
\[
\begin{bmatrix}
    \rho & \rho_{\overline{\beta}} \\
\rho_{\alpha}\left(1+\frac{2}{n} q \rho^2\right)+o\left(\rho^2\right) & \rho_{\alpha \overline{\beta}}-\frac{2}{n}\left(q_{\overline{\beta}} \rho_{\alpha}+q \rho_{\alpha \overline{\beta}}\right) \rho^2+o\left(\rho^2\right)
\end{bmatrix}.
\]
Now add $\frac{2}{n} q_{\bar{\beta}} \rho^2$ times the first column to the $(\beta+1)$-th column, then 
\[
   J[\Phi]
=(-1)^{n}\left(1-\frac{n+1}{n}q \rho^2\right)\det\begin{bmatrix}
    \rho & \rho_{\bar\beta}\\
    \rho_{\alpha}\left(1+\frac{2}{n} q\rho^2\right) & \rho_{\alpha \overline{\beta}}(1-\frac{2}{n}q \rho^2 )
\end{bmatrix}+o(\rho^{2}).
\]
Factoring $\left(1+\frac{2}{n} q \rho^2\right)$ from the first column and $\left(1-\frac{2}{n} q \rho^2\right)$ from the other $n$ columns gives
\begin{align*}
    J[\Phi]&=(-1)^{n}(1-\frac{n+1}{n}q\rho^{2})(1+\frac{2}{n}q\rho^{2})(1-\frac{2}{n}q\rho^{2})^{n}\det\begin{bmatrix}
        \rho\left(1+\frac{2}{n} q \rho^2\right)^{-1} & \rho_{\bar\beta}\left(1-\frac{2}{n} q \rho^2\right)^{-1}\\
        \rho_{\alpha} & \rho_{\alpha\bar\beta}
    \end{bmatrix}+o(\rho^{2}).
\end{align*}
Finally, factoring $\left(1-\frac{2}{n} q \rho^2\right)^{-1}$ from the first row yields
\begin{align*}
    J[\Phi]&=\left(1-\frac{n+1}{n}q \rho^2\right)\left(1+\frac{2}{n} q \rho^2\right)\left(1-\frac{2}{n} q \rho^2\right)^{n-1} J[\rho]+o\left(\rho^2\right).
\end{align*}
\end{proof}
\begin{proof}[Proof of Theorem \ref{gen_rama_conj}]
Recall that $\rho$ is a Fefferman defining function, hence $J(\rho)=1+O(\rho^{n+1})$. Since $n\geq 4$, Lemma \ref{approx:J_exp} implies
\[
J[\Phi]=\left(1-\frac{n+1}{n}q \rho^2\right)\left(1+\frac{2}{n} q \rho^2\right)\left(1-\frac{2}{n} q \rho^2\right)^{n-1}+o\left(\rho^2\right).
\]
Using the expansion $(1+x)^{n-1}=1+(n-1) x+O(x^2)$, we get
\[
\begin{aligned}
J[\Phi]=1+ & \left(-\frac{n+1}{n}+\frac{2}{n}-2 \frac{n-1}{n}\right) q \rho^2+o\left(\rho^2\right) \\
& =1+ \left(\frac{3-3n}{n}\right) q \rho^2+o(\rho^2).
\end{aligned}
\]
On the other hand,
\begin{multline*}
    J[\Phi]=J\Big(\frac{C_{n}}{n^{n}}S^{\frac{n+1}{n}}\Big)^{-\frac{1}{n+1}}=\left(\frac{C_{n}}{n^{n}}S^{\frac{n+1}{n}}\right)^{-1}n^{-n}\det G(z)
    =\frac{n^{-n}\det G(z)}{C_{n}n^{-n}S^{\frac{n+1}{n}}}=\frac{\mathfrak R(z)}{C_{n}}.
\end{multline*}
Consequently,
\[
\mathfrak{R}-C_{n}=\left(\frac{3-3n}{n}\right)C_{n} q \rho^2+o\left(\rho^2\right).
\]
Therefore 
\[\lim_{z\to p} \rho^{-2}(z)\left(\mathfrak{R}(z)-C_{n}\right)=\left(\frac{3-3n}{n}\right)C_{n}q(p).\]
    
By (\ref{asym:phi_n_ge_3}), the function $q$ is a nonzero multiple of the Chern–Moser curvature invariant $|A^0_{2\overline{2}}|^2$. Hence $q(p)=0$ if and only if $p$ is a CR umbilical point of $\pa\Om$. This proves the theorem.
\end{proof}
\begin{corollary}
    Let $\Om=\{\rho<0\}\Subset\mbb C^n,n\geq 4$, be a $C^\infty$-smooth strongly pseudoconvex domain with Fefferman defining function $\rho$, and let $
\mathfrak{R}_\Om(z)$ be the Fefferman-Szeg\H{o} invariant function. Then,
\[\mathfrak{R}_\Om-C_{n}=\left(\frac{3-3n}{n}\right)C_{n} q \rho^2+o\left(\rho^2\right),\] where $C_n=((n-1)!)^{-\frac{n+1}{n}}n^n\pi^{n+1}$.
\end{corollary}

\section{K{\"a}hler immersions}\label{kahler immersion}

A \emph{complex space form} is a connected complete K{\"a}hler manifold whose holomorphic sectional curvature is constant, say $2b$. Up to biholomorphic isometries, the simply connected ones fall into three types depending on the sign of $b$ (see \cite{l66}, Theorem $1$):
\begin{itemize}
    \item \textbf{Flat case} (\(b = 0\)): the complex Euclidean space \(\mathbb{C}^N\) (with \(N \le \infty\)), with the flat metric \(g_0\).  
          Here \(\mathbb{C}^\infty\) denotes the Hilbert space of square‑summable complex sequences.

    \item \textbf{Hyperbolic case} (\(b < 0\)): the complex hyperbolic space \(\mathbb{CH}^N:=\mathbb B^N_{(-b)^{-1}}\), which is the ball of radius \((-b)^{-1}\) in \(\mathbb{C}^N\) (or \(\mathbb{C}^\infty\)), with the hyperbolic metric $(-b)^{-1}g_{\text{hyp}}$.  
          Its Kähler form is given by
          \[
          \omega_{\text{hyp}} = \frac{1}{b}\sqrt{-1}\,\partial\bar\partial\log\bigl(1 + b\|z\|^2\bigr).
          \]

    \item \textbf{Projective case} (\(b > 0\)): the complex projective space \(\mathbb{CP}^N\) (with \(N \le \infty\)), with the Fubini–Study metric \(\frac{1}{b}g_{\operatorname{FS}}\).  
          In homogeneous coordinates $[Z_0,\ldots,Z_N]$ coordinates, its Kähler form is given by
          \[
          \omega_{FS} = \frac{1}{b}\sqrt{-1}\,\partial\bar\partial\log\Bigl(1 + \sum_{i \neq j} \bigl|Z_i/Z_j\bigr|^2\Bigr).
          \]
\end{itemize}
Here, $\mathbb CP^\infty$ is the projectivization of $\mathbb C^\infty$. We denote these spaces by $F(N, b)$, called \emph{Fubini--Study spaces}. In \cite{ca53}, Calabi study the  K{\"a}hler immersions into $F(N, b)$ using the diastasis function which is defined as follows. Let  $(M,g)$ be a real-analytic K{\"a}hler manifold and fix $p\in M$. Choose a local real analytic K{\"a}hler potential $\phi$ for $g$ in coordinate $\varphi$ near $p$. The \emph{diastasis function} near $p$ using analytic continuation is defined as
\begin{align*}
    \mathfrak{D}^M(z, w) = \phi(z, \bar{z}) + \phi(w, \bar{w}) - \phi(z, \bar{w}) - \phi(w, \bar{z}).
\end{align*}
 If $\varphi(p) = w_0$, the \emph{diastasis potential} at $p$ is $\mathfrak D_p^M(z) = \mathfrak D^M(z, w_0)$
which is again a real-analytic K{\"a}hler potential near $p$. It is independent of the choice of potential and coordinates. A holomorphic map $f \colon M \to N$ of real-analytic K{\"a}hler manifolds is an isometry if and only if it preserves the diastasis potential (see \cite{ca53}, Proposition $1-6$), i.e.,
\begin{equation}\label{Eqdiastasispreserved}
    f^*g^N = g^M \quad \text{if and only if}\quad \forall\ p \in M:\ \mathfrak D_p^M = \mathfrak D^N_{f(p)} \circ f.
\end{equation}
For a modern exposition; see \cite{LM18}. In this section, the following examples will be useful.
\begin{example}\label{Ex1}
It can be seen that the K{\"a}hler manifold $(\Omega, g_{\operatorname{FS}}^{\Om})$ is real analytic. For $p \in \Omega$, the diastasis around $p$ is globally defined on $\Omega$, given by
\begin{equation}\label{EqdiastasisBergman}
    \mathfrak D^\Omega(z, w) = \log\left( \frac{S_\Omega(z) S_\Omega(w)}{|S_\Omega(z, w)|^2}\right).
\end{equation}
\end{example}
\begin{example}\label{Ex2}
For $p \in F(N,b)$ with $b \neq 0$, the diastasis around $p$ in Bochner coordinates centered at $p$ is
\begin{equation}\label{Eqdiastasishyperbolic}
    \mathfrak D^b(z, w) = \frac{1}{b} \log\left(\frac{(1 + b \left\|z\right\|^{2})(1 + b \left\|w\right\|^{2})}{|1 + b\langle z,w \rangle|^{2}}\right).
\end{equation}
In the case $b > 0$, the diastasis around any point can be expressed in homogeneous coordinates $Z = [Z_0, \dots, Z_N],\ W = [W_0, \dots, W_N]$ as
\begin{equation}\label{Eqdiastasisprojective}
    \mathfrak D^b(Z, W) = \frac{1}{b} \log \left( \frac{(\sum_{i = 0}^N | Z_i |^2 )(\sum_{i = 0}^N | W_i |^2 )}{(\sum_{i = 0}^N Z_i \overline{W}_i )^{2}} \right)
\end{equation}
\end{example}
\begin{example}\label{Ex3}
Let $(M, g)$ be a real analytic K{\"a}hler manifold, then for every $ \lambda > 0$, $(M, \lambda g)$ is also real analytic. Then, we have
\begin{equation}\label{Eqdiastasisrescaling}
   \mathfrak D^{\lambda g} = \lambda \mathfrak D^g.
\end{equation}
\end{example}

\subsection*{Transversality}
Let $\Omega_1=\{z\in U_1:\rho_1(z)<0\}\subset\mathbb{C}^n$ and $\Omega_2=\{z\in U_2:\rho_2(z)<0\}\subset\mathbb{C}^m$ be bounded domains with $C^1$-regular boundaries where each $U_{i}$ is an open neighborhood of $\overline{\Omega}_i$. A map $F\in C^1(U_1, U_2)$ with $F(\partial \Omega_1) \subset \partial \Omega_2$ is \emph{transverse} to $\partial \Omega_2$ at $x \in \partial \Omega_1$ if
\begin{align*}
    (D_xF)(T_x\Omega_1) + T_{F(x)} \partial \Omega_2 = T_{F(x)} \Omega_2.
\end{align*}
Equivalently,
 \begin{align*}
    \langle (D_xF)(\nabla \rho_1(x)), \nabla \rho_2(F(x)) \rangle \neq 0,
\end{align*}
This implies $\langle \nabla \rho_1(x), \nabla(\rho_2 \circ F)(x) \rangle\neq 0$ so $\nabla(\rho_2 \circ F)(x)$ is nowhere zero on $\pa\Om_1$. $F$ is said to be \emph{transverse along} $\partial \Omega_2$ if it is transverse to $\partial \Omega_2$ at any point of $\partial \Omega_1$.

\medskip 

Following the ideas from \cite{p25}, we finally proof Theorem \ref{kah_imm_ball}.
\begin{proof}[Proof of Theorem \ref{kah_imm_ball}]
    We work in any dimension $n$ first. After translation and rotation, assume $0 \in \Omega$ and $f(0) = 0$. Since $(\mathbb{B}^N, g_{\operatorname{FS}}^{\mathbb{B}^N}) \equiv F(N, -\frac{1}{N})$, we choose the defining function for $F(N, -\frac{1}{N})$:
\begin{align*}
    \rho_{\mathbb{B}^N}(z) = 1 - \frac{1}{N} || z ||^2\ ,\quad z \in F(N, -\frac{1}{N}).
\end{align*}
By \eqref{EqdiastasisBergman}, \eqref{Eqdiastasishyperbolic}, \eqref{Eqdiastasisrescaling}, we have 
$$f^{*}g_{\operatorname{FS}}^{\mathbb{B}^{N}}=\lambda g_{\operatorname{FS}}^{\Omega},\quad f^{*}\mathfrak D^{g_{\operatorname{FS}}^{\mathbb{B}^{N}}}=\mathfrak D^{\lambda g_{\operatorname{FS}}^{\Omega}},\quad \mathfrak D^{\lambda g_{\operatorname{FS}}^\Omega} = \lambda \mathfrak D^{g_{\operatorname{FS}}^\Omega}.$$
Then for any $z\in\Omega$,
$$\lambda \mathfrak D^{\Omega}(z,0)=\mathfrak D^{\mathbb{B}^N}(f(z),0),$$
where
$$\mathfrak D^{\Omega}(z,0) =\log\left( \frac{S_\Omega(z) S_\Omega(0)}{|S_\Omega(z, 0)|^2}\right)$$
and 
$$\mathfrak D^{\mathbb{B}^N}(f(z),0)=-N\log \left(1-\frac{1}{N}\left\|f(z)\right\|^{2}\right)=\log\frac{1}{\rho_{B^N}(f(z))^{N}}.$$
Hence \eqref{Eqdiastasispreserved} implies
\begin{equation}\label{EqdiastasisFefferman}
    S_\Omega(z) = \frac{| S_\Omega(z, 0) |^2}{S_\Omega(0)} \frac{1}{\rho_{\mathbb{B}^N}(f(z))^{\frac{N}{\lambda}}}\ ,\quad z \in \Omega.
\end{equation}
Because $F(\Omega) = f(\Omega) \subset \mathbb{B}^N$,$F(\partial \Omega) \subset \partial \mathbb{B}^N$,$F(U \setminus \overline\Omega) \subseteq V \setminus\overline{\mathbb{B}^N}$, the function $\rho_\Omega := \rho_{\mathbb{B}^N} \circ F \in C^\infty(\overline{\Omega})$ satisfies
\begin{align*}
    \Omega = \{\rho_\Omega > 0\}\ ,\quad \partial \Omega = \{ \rho_\Omega = 0 \},
\end{align*}
Furthermore, from transversality of $F$, we deduce that $\nabla \rho_\Omega |_{\partial \Omega} \neq 0$. Hence, $\rho_\Omega\in C^\infty(\ov\Om)$ is a defining function for $\Omega$ which implies that $\Omega$ is $C^\infty$-smoothly bounded. We next recall the Fefferman asymptotic expansion formula (see \cite{fc74,bs76}) as follows
\begin{equation}\label{asym_exp:sgo}
    S_\Om(z)=\frac{\phi_\Om(z)}{\rho_\Om(z)^n}+\psi_\Om(z)\log\rho_\Om(z)
\end{equation}
with $\phi_\Om,\psi_\Om\in C^\infty(\ov\Om)$ and $\phi|_{\pa\Om}\neq 0$.
Using \eqref{asym_exp:sgo} in \eqref{EqdiastasisFefferman} we get
\begin{align*}
    \frac{| S_\Omega(z, 0) |^2}{S_\Omega(0)} - \phi_\Omega(z) \rho_\Omega(z)^{\frac{N}{\lambda} - n} - \rho_\Omega(z)^{\frac{N}{\lambda}} \psi_\Omega(z) \log \rho_\Omega(z) = 0\ ,\quad z \in \Omega.
\end{align*}

Since $S_\Omega(\cdot,0)\in C^\infty(\overline{\Omega})$; see \cite{bs76}, Lemma 2.2 of \cite{fw97} shows that, whenever
\[
\frac{N}{\lambda}-n\in\mathbb N,
\]
the function $\rho_\Omega^{N/\lambda}\psi_\Omega$ vanishes to infinite order on $\partial\Omega$. As $\nabla\rho_\Omega\neq 0$ on $\partial\Omega$, this implies that $\psi_\Omega$ vanishes to infinite order on $\partial\Omega$. For $n=2$, such vanishing forces $\partial\Omega$ to be locally spherical by \cite{hkn93}. Hence, if $\Omega\subset\mathbb C^2$ is simply connected, \cite{CJ96} implies that $\Omega$ is biholomorphic to $\mathbb B^2$. This completes the proof.

\end{proof}

\subsection*{Acknowledgements}
The first author expresses her sincere gratitude to Professor Kengo Hirachi for bringing the relevant reference to her attention. The authors are deeply grateful to Professor Xiaoshan Li for his constant encouragement and support during the preparation of this work. They also gratefully acknowledge the financial support provided by the Wuhan University. The first author was partially supported by NSFC (12361131577), and the second author by NSFC (12271411), through Professor Xiaoshan Li.


\end{document}